\documentclass[a4paper]{article}

\usepackage{geometry}
\usepackage{amsmath,amssymb,amsfonts,mathtools}
\usepackage{amsthm}

\usepackage{indentfirst}
\usepackage{array}
\usepackage{color}
\usepackage{cite}
\usepackage{microtype}
\usepackage{enumitem}
\usepackage{needspace}
\usepackage{tikz}

\usepackage[colorlinks=true]{hyperref}

\allowdisplaybreaks[1]
\numberwithin{equation}{section}

\theoremstyle{plain}
\newtheorem{theorem}{Theorem}[section]
\newtheorem{lemma}[theorem]{Lemma}
\newtheorem{proposition}[theorem]{Proposition}
\newtheorem{corollary}[theorem]{Corollary}
\newtheorem{claim}[theorem]{Claim}
\theoremstyle{remark}
\newtheorem{remark}[theorem]{Remark}

\DeclareMathOperator{\Tr}{Tr}
\newcommand{\C}{\mathbb C}
\newcommand{\R}{\mathbb R}
\newcommand{\N}{\mathbb N}
\newcommand{\calL}{\mathcal L}
\newcommand{\Sch}{\mathcal S}
\newcommand{\dd}{\,d}
\newcommand{\br}[1]{\langle #1\rangle}
\DeclarePairedDelimiter{\norm}{\lVert}{\rVert}
\newcommand{\ind}{\mathbf 1}
\newcommand{\ad}{\operatorname{ad}}
\newcommand{\Spec}{\operatorname{Spec}}
\newcommand{\Ran}{\operatorname{Ran}}
\newcommand{\supp}{\operatorname{supp}}
\newcommand{\sgn}{\operatorname{sgn}}
\newcommand{\ee}{\epsilon}
\newcommand{\Ht}{\widetilde H}
\newcommand{\Pt}{\widetilde P}
\newcommand{\Vt}{\widetilde V}
\newcommand{\Xt}{X^{\mathrm R}}
\newcommand{\Yt}{Y^{\mathrm R}}
\newcommand{\FA}{\mathcal F}
\newcommand{\Hs}[1]{H^{#1}}
\newcommand{\sloc}{\mathrm{sloc}}
\newcommand{\ip}[2]{\langle #1,#2\rangle}
\begin{document}

\title{A sharp spectral multiplier theorem on the Engel group}

\author{Zhicheng Han and Zhipeng Yang\thanks{Corresponding author: yangzhipeng326@163.com}}

\date{}
\maketitle

\newcommand{\authoraddresses}{
  \par
  \medskip
  \noindent
  \begin{minipage}{\textwidth}
    \linespread{1}\small
    \textbf{Zhicheng Han:}\\[0.2em]
    \textsc{Institut f\"ur Analysis, Leibniz Universit\"at Hannover, 30167 Hannover, Germany}\\[0.3em]
    \textit{E-mail address}: \texttt{hanzc@math.uni-hannover.de}\\[0.8em]
    
    \noindent
    \textbf{Zhipeng Yang:}\\[0.2em]
    \textsc{Department of Mathematics, Yunnan Key Laboratory of Modern Analytical Mathematics and Applications, Yunnan Normal University, 650500 Kunming, China}\\
    \textsc{Department of Mathematics: Analysis, Logic and Discrete Mathematics, Ghent University, 9000 Ghent, Belgium}\\[0.3em]
    \textit{E-mail address}: \texttt{yangzhipeng326@163.com}
  \end{minipage}
}

\begin{abstract}
We prove a spectral multiplier theorem for the standard sub-Laplacian on the Engel group under a scale-invariant local $H^s$ condition with $s>2$. The multiplier operators $F(\mathcal L)$ are of weak type $(1,1)$ and bounded on $L^p$ for every $1<p<\infty$. Together with the known lower bound, this identifies the sharp threshold as $2$, half the topological dimension. The proof combines a weighted Plancherel estimate for quartic oscillators with a decomposition into rank-two spectral subspaces, which controls parameter derivatives uniformly as one well escapes to infinity. For fixed smooth multipliers compactly supported in $(0,\infty)$, we also determine the leading term and first correction for the localized kernels. Whether the weak type $(1,1)$ estimate holds under the local $H^2$ condition remains open.


\end{abstract}

\section{Introduction and main results}\label{Sec1}
For a nonnegative self-adjoint operator $\calL$, the spectral theorem gives $\norm{F(\calL)}_{2\to2}\le\norm F_\infty$. For $p\ne2$, bounds for $F(\calL)$ on $L^p$ require regularity of $F$ and apply, for instance, to the imaginary powers of $\calL$. For the Laplacian on $\R^d$,
\[
\widehat{F(-\Delta)f}(\xi)=F(|\xi|^2)\widehat f(\xi),
\]
so the problem is part of the classical theory of Fourier multipliers. Fix a nonzero $\chi\in C_c^\infty((0,\infty))$ and, for $s\ge0$, write
\begin{equation}\label{eq1.1}
\norm F_{H^s,\sloc}=\sup_{r>0}\norm{\chi(\cdot)F(r\cdot)}_{H^s(\R)}.
\end{equation}
Here $H^s=H^s(\R)$ is the $L^2$ Sobolev space with norm $\norm f_{H^s}=\norm{(1+|\xi|^2)^{s/2}\widehat f(\xi)}_2$, for a fixed Fourier normalization. The subscript $\sloc$ stands for ``scale-invariant local''. Different nonzero cutoffs give equivalent norms. In its spectral form for the Euclidean Laplacian, the classical Mihlin--H\"ormander theorem~\cite{Hor} states
\[
s>\frac d2
\quad\Longrightarrow\quad
\norm{F(-\Delta)}_{L^1\to L^{1,\infty}}
+\norm{F(-\Delta)}_{L^p\to L^p}
\lesssim_{s,p}\norm F_{H^s,\sloc},
\qquad 1<p<\infty.
\]

On a stratified group with Lie algebra $\mathfrak g=V_1\oplus\cdots\oplus V_r$, the two relevant dimensions are
\[
d=\sum_{j=1}^r\dim V_j,\qquad Q=\sum_{j=1}^r j\dim V_j.
\]
The sub-Laplacian $\calL=-\sum_{j=1}^{\dim V_1}X_j^2$ is homogeneous of degree two, whereas Haar measure scales with degree $Q$. Control balls of radius $R$ have volume proportional to $R^Q$. The general multiplier theorems of Christ and Mauceri-Meda~\cite{Chr,MauMed} give the same weak type and $L^p$ bounds as above under the stronger condition $s>Q/2$. Since $Q>d$ on every nonabelian stratified group, these theorems leave a gap between Euclidean regularity and the regularity given by homogeneous volume growth.

Write $\varsigma(\calL)$ for the infimum of all $s\ge0$ such that, for every $1<p<\infty$, there is a constant $C_{p,s,\chi}$ satisfying
\[
\norm{F(\calL)}_{L^p\to L^p}\le C_{p,s,\chi}\norm F_{H^s,\sloc}
\]
for every bounded Borel function $F:[0,\infty)\to\C$. Martini, M\"uller and Nicolussi Golo~\cite{MMN} showed, for sub-Laplacians on stratified groups of arbitrary step, that a multiplier estimate on $L^p$ with a local $L^\infty$ Sobolev norm of order $a$ requires
\[
a\ge d\left|\frac1p-\frac12\right|,
\]
and that a threshold defined for bounded continuous multipliers with a norm closely related to \eqref{eq1.1} is at least $d/2$. The comparison in Section~\ref{Sec6} transfers this lower bound to $\varsigma(\calL)$. Thus the general bounds are
\[
\frac d2\le\varsigma(\calL)\le\frac Q2.
\]
Determining where the threshold lies in this interval requires information about the spectral resolution beyond homogeneous volume growth.

Sharp results in this direction were obtained on Heisenberg and generalized Heisenberg groups by Hebisch and by M\"uller-Stein~\cite{Heb,MS}. For the standard Heisenberg group $\mathbb H^n$, they establish
\[
\varsigma(\calL)=\frac{2n+1}{2},\qquad d=2n+1,\quad Q=2n+2.
\]
The topological threshold has since been established for further classes of groups. Martini and M\"uller obtained the sufficient condition $s>3$ on the free two-step group $N_{3,2}$, for which $d=6$ and $Q=9$~\cite{MMfree}. Together with the lower bound, this yields $\varsigma(\calL)=3$. Their results for further two-step groups cover all groups with $\dim V_2\le2$ and all groups with $d\le7$~\cite{MMnew}. Martini obtained the $d/2$ threshold for suitable sub-Laplacians on Heisenberg-Reiter groups and related classes~\cite{MarHR}. For an arbitrary nonabelian two-step stratified group, Martini and M\"uller proved the strict improvement
\[
\frac d2\le\varsigma(\calL)<\frac Q2
\]
without additional algebraic assumptions~\cite{MMtop}. Whether $\varsigma(\calL)=d/2$ on every stratified group remains open, even in step two~\cite{MarHR,MMtop}. More recent work of Niedorf~\cite{Nie} obtains the sharp regularity condition $s>d|1/p-1/2|$ on M\'etivier groups for ranges of $p$ determined by the layer dimensions. Sharp wave estimates on groups of Heisenberg type provide a related instance in which topological dimension determines the regularity loss~\cite{MSwave}.

We consider the connected, simply connected Engel group $G$, whose Lie algebra has basis $X,Y,Z,W$ and nonzero brackets
\[
[X,Y]=Z,\qquad [X,Z]=W.
\]
Its layers are $\operatorname{span}\{X,Y\}$, $\R Z$ and $\R W$, so $d=4$ and $Q=7$. We use the nonnegative self-adjoint realization of $\calL=-X^2-Y^2$ associated with the horizontal Dirichlet form. The general bounds leave the interval
\[
2\le\varsigma(\calL)\le\frac72.
\]
On Heisenberg groups the Fourier transform reduces the sub-Laplacian to harmonic oscillators. On the Engel group it leads to quartic oscillators, whose potentials may have two wells. Controlling these two wells is the main new difficulty.

Bahouri, Barilari, Gallagher and L\'eautaud~\cite{BBGL} developed Fourier inversion and Plancherel formulas for the Engel group and established the spectral summability needed to use them. If $E_m(\mu)$ are the eigenvalues of
\[
-\frac{d^2}{du^2}+\left(\frac{u^2}{2}-\mu\right)^2,
\]
their Theorem~1.3 states
\[
\sum_{m=0}^\infty\int_{\R}E_m(\mu)^{-q}\dd\mu<\infty\quad\Longleftrightarrow\quad q>2.
\]
For smooth $F$ compactly supported in $(0,\infty)$, their Proposition~1.10 states the kernel identity
\[
\norm{K_F}_2^2=c_{\calL}\int_0^\infty|F(E)|^2E^{5/2}\dd E,
\]
where $K_F$ is the convolution kernel of $F(\calL)$. Related work treats $L^p$-$L^q$ estimates for other Engel and Cartan operators~\cite{Cha} and the Engel Schr\"odinger flow~\cite{Ben}. To pass from the spectral estimates above to a sharp multiplier theorem, we need spatially weighted kernel estimates and control of derivatives with respect to the Fourier parameters.

Weighted Plancherel estimates are central both to abstract spectral multiplier theorems~\cite{DOS} and to sharp results for Grushin operators~\cite{MSGr,MMGr}. In the plane, the reduction
\[
-\partial_x^2-V(x)\partial_y^2\quad\longmapsto\quad-\partial_x^2+\xi^2V(x)
\]
connects the multiplier problem with one-dimensional spectral analysis. Dall'Ara and Martini developed this approach for increasingly general potentials~\cite{DMrobust,DM,DMII}. For a class of doubling potentials with a single well, they first obtained the optimal order $s>1$ with a local $L^\infty$ Sobolev condition~\cite{DM}, and then with the local $H^s$ condition~\cite{DMII}. Their estimates for eigenvalue gaps provide one ingredient of our proof. On the Engel group, we also need uniform control as two wells separate and one of them escapes to infinity.

\begin{theorem}\label{Thm1.1}
For every $s>2$, every bounded Borel function $F:[0,\infty)\to\C$ satisfies
\[
\norm{F(\calL)}_{L^1(G)\to L^{1,\infty}(G)}\le C_{s,\chi}\norm F_{H^s,\sloc}.
\]
For every $1<p<\infty$, one also has
\[
\norm{F(\calL)}_{L^p(G)\to L^p(G)}\le C_{p,s,\chi}\norm F_{H^s,\sloc}.
\]
Moreover,
\begin{equation}\label{eq1.2}
\varsigma(\calL)=2.
\end{equation}
\end{theorem}

Theorem~\ref{Thm1.1} replaces the general sufficient condition $s>7/2$ by $s>2$, with weak type $(1,1)$ and boundedness throughout the range $1<p<\infty$. To our knowledge, the Engel group is the first stratified group of step three for which this threshold is known. The local norm depends only on the restriction of $F$ to $(0,\infty)$ modulo null sets. Lemma~\ref{Lem5.4} shows that spectral projections of Lebesgue null subsets of $(0,\infty)$ vanish, as does the spectral projection at zero, so the theorem applies to arbitrary bounded Borel representatives.

The proof rests on a weighted kernel estimate. We use the coordinates $(x,y,z,w)=\exp(yY+zZ+wW)\exp(xX)$, with Haar measure $dx\,dy\,dz\,dw$, and put
\[
\br v=1+|v|,
\qquad
\rho(x,y,z,w)=|x|+|y|+|z|^{1/2}+|w|^{1/3}.
\]
The gauge $\rho$ is comparable with the control distance from the identity. The coordinates $z$ and $w$ have homogeneous degrees two and three, respectively.

\begin{theorem}\label{Thm1.2}
Let $J\Subset(0,\infty)$ be a compact interval. For $\alpha,\beta,\gamma\ge0$ with $\beta+2\gamma<3/2$ and $\kappa>0$,
\begin{equation}\label{eq1.3}
\norm{\br\rho^\alpha\br z^\beta\br w^\gamma K_F}_2
\le C_{J,\alpha,\beta,\gamma,\kappa}
\norm F_{H^{\alpha+\beta+\gamma+\kappa}}
\end{equation}
for $F\in C_c^\infty(J)$. The estimate extends to $F\in H^{\alpha+\beta+\gamma+\kappa}(\R)$ with $\supp F\subset J$, with the kernel defined by the corresponding weighted $L^2$ limit.
\end{theorem}

By Corollary~\ref{Cor6.1}, we also have
\[
\norm{\br\rho^\delta K_F}_1
\le C_{s,\delta,J}\norm F_{H^s},
\qquad s>2,\quad 0\le\delta<s-2.
\]
For fixed smooth $F$, let $K_{F,h}$ be the localized kernels defined in Section~\ref{Sec2.1}. Proposition~\ref{Prop5.8} identifies the limit of $h^{-3}K_{F,h}(x,y,z/h,w/h^2)$ as the product of the Euclidean kernel of $F(-\Delta)$ on the abelian quotient $\R^2$ and the inverse Fourier transform of the localization cutoff. The first correction and exact leading constants for all integer moments identify the powers of $h$ in Proposition~\ref{Prop5.2} and remove its logarithmic factor for fixed smooth multipliers.

The equality \eqref{eq1.2} and the lower bound from \cite{MMN} concern the infimum of the admissible Sobolev orders. The estimate at the endpoint $s=2$ remains open.

\medskip
\noindent\textbf{Open problem.} Is there a constant $C_\chi$ such that
\[
\norm{F(\calL)}_{L^1(G)\to L^{1,\infty}(G)}
\le C_\chi\norm F_{H^2,\sloc}
\]
for every bounded Borel function $F:[0,\infty)\to\C$ with $\norm F_{H^2,\sloc}<\infty$?
\medskip

A second question is whether the condition $\beta+2\gamma<3/2$ in Theorem~\ref{Thm1.2} is optimal for the global weighted estimate. The localized asymptotics in Proposition~\ref{Prop5.8} determine the powers of $h$ for each fixed smooth multiplier, but do not settle this global question.

For Theorem~\ref{Thm1.1}, we need only the diagonal choice $\beta=\gamma$, with both exponents approaching $1/2$ from below.

The threshold $2$ follows from the balance between the kernel weights and the required regularity of $F$. For $0\le\beta,\gamma<1/2$,
\[
\br\rho^{-\alpha}\br z^{-\beta}\br w^{-\gamma}\in L^2(G)
\quad\text{if}\quad
\alpha+2\beta+3\gamma>\frac72,
\]
whereas \eqref{eq1.3} requires only $\alpha+\beta+\gamma+\kappa$ derivatives of $F$. As $\beta$ and $\gamma$ approach $1/2$ from below, this order approaches $2$. We obtain the homogeneous weight from finite propagation speed~\cite{CM}, then apply weighted Cauchy-Schwarz and Calder\'on-Zygmund theory~\cite{CW,FS} to deduce the multiplier bounds. Theorem~\ref{Thm1.2} controls the $z$ and $w$ moments with fewer derivatives of $F$ than their homogeneous degrees $2$ and $3$ would suggest; this is the source of the gain over $s>7/2$.

Partial Fourier transformation in $(y,z,w)$ reduces the problem to
\[
H_{\eta,\zeta,\lambda}
=-\partial_x^2+\bigl(\eta+\zeta x+\lambda x^2/2\bigr)^2.
\]
The $z$ and $w$ moments correspond to derivatives in $\zeta$ and $\lambda$. We divide the Fourier parameter space into two regions, according to whether a well can escape to infinity. In the region away from this regime, we estimate the parameter derivatives using spectral clusters and iterated commutators. The estimates have polynomial losses in the inverse quadratic coefficient and a logarithmic loss at critical energies. We absorb the polynomial losses using the decay in the parameters of the kernel column $\psi(H)\delta_0$ of a smooth spectral cutoff $\psi$. To keep the required regularity in $L^2$ Sobolev spaces, we average divided differences over spectral dilations. For $f_\tau(E)=F(\tau^2E)$ and nonnegative spectral nodes, Lemma~\ref{Lem2.3} states
\[
\left(\int_1^2|f_\tau^{[k]}(E_0,\ldots,E_k)|^2\dd \tau\right)^{1/2}
\le C_{J,k}\norm F_{H^k},\qquad k\in\N_0,
\]
where $f_\tau^{[k]}$ is the divided difference of order $k$ and $F$ is supported in a compact interval $J\subset(0,\infty)$. We use this estimate for $k\le2$. After integration over the Fourier parameters, we use homogeneity to pass from the averaged estimates to bounds for a fixed multiplier, with the Sobolev orders in \eqref{eq1.3}.

The remaining region contains an escaping well. For fixed $\eta$ and $\zeta\ne0$, as $\lambda\to0$ one root of $\eta+\zeta x+\lambda x^2/2$ tends to $-\eta/\zeta$, while the other is $-2\zeta/\lambda+O(1)$. Both wells can contribute at energies in $J$. For kernel columns whose source variable lies near the bounded well, we compare the full kernel with that of a uniformly confining single-well operator. The eigenvalue gaps in~\cite[Theorem 7.6]{DM} control the separation between the spectral pairs in this comparison. Each eigenvalue of the comparison operator below the fixed energy cutoff corresponds to a rank-two spectral subspace of the full operator. In a frame adapted to reflection, a fixed matrix diagonalizes the restriction of the full operator to each such subspace. We control parameter derivatives using the separation between neighboring pairs, independently of the small splitting within a pair. By Agmon estimates and spectral projections~\cite{Agm,HS,Sim}, the errors in the averaged weighted $L^2$ comparison of these columns are exponentially small. These bounds are uniform over all eigenvalues below a fixed energy and include parameter derivatives up to order two.

Combining the estimates on these two regions, we obtain, for kernels $K_{F,h}$ localized where $|\zeta|+|\lambda|^{1/2}\asymp h$,
\[
\norm{z^kw^lK_{F,h}}_2\lesssim h^{3/2-k-2l}(1+|\log h|)^2
\norm F_{H^{k+l+\kappa}},\qquad k,l\in\N_0,\quad k+l\le2,\quad 0<h\le1.
\]
Interpolation between order zero and the two pure second-order estimates produces $h^{3/2-\beta-2\gamma}$. Summation over dyadic $h$ requires $\beta+2\gamma<3/2$, the range in Theorem~\ref{Thm1.2}.

Section~\ref{Sec2} establishes the Fourier conventions, estimates for the kernel columns of smooth cutoffs, and the averaged functional calculus, using Hulanicki's theorem~\cite{Hul} and the resolvent calculus~\cite[Chapter 8]{DS}. The spectral estimates in Section~\ref{Sec3} and the comparison with the escaping well in Section~\ref{Sec4} then yield the weighted kernel bounds. Section~\ref{Sec5} combines these estimates to prove Theorem~\ref{Thm1.2} and also derives the localized kernel asymptotics. Section~\ref{Sec6} uses the weighted bounds to prove Theorem~\ref{Thm1.1}.

\section{Fourier reduction and averaged functional calculus}\label{Sec2}
Throughout the paper, $C$, $c$ and $M$ denote positive constants that may change from line to line. Inner products are linear in the first variable. We denote by $\FA$ the Euclidean Fourier transform in $(y,z,w)$, with dual variables $(\eta,\zeta,\lambda)$ and phase $e^{-i(y\eta+z\zeta+w\lambda)}$. For a one-dimensional operator $T$, $k_T(u,v)$ denotes its integral kernel with respect to Lebesgue measure, and $u\mapsto k_T(u,v)$ at fixed $v$ is called a kernel column. Unless an extension by density is stated, kernel estimates are first proved for $F\in C_c^\infty(J)$, where $J\Subset(0,\infty)$ is a compact interval. For estimates stated with an arbitrary $\kappa>0$, choose $0<\kappa_0<\min\{\kappa,1\}$ and apply the estimate with loss $\kappa_0$. The bound with loss $\kappa$ follows from
\[
\norm F_{H^{r+\kappa_0}}\le\norm F_{H^{r+\kappa}},\qquad r\ge0.
\]

\subsection{Fourier reduction and kernel columns}\label{Sec2.1}
In the coordinates of Section~\ref{Sec1}, the group law is
\begin{equation}\label{eq2.1}
\begin{split}
(x,y,z,w)(x',y',z',w')
=\bigl(x+x',\ y+y',\ z+z'+xy',\
w+w'+xz'+\tfrac12x^2y'\bigr).
\end{split}
\end{equation}
Lebesgue measure is Haar measure. The left-invariant fields $X,Y,Z,W$ and the right-invariant horizontal fields $\Xt,\Yt$ are
\begin{align*}
X&=\partial_x,&Y&=\partial_y+x\partial_z+\tfrac12x^2\partial_w,
&Z&=\partial_z+x\partial_w,&W&=\partial_w,\\
\Xt&=\partial_x+y\partial_z+z\partial_w,
&\Yt&=\partial_y.&&&
\end{align*}
The dilations $\delta_r(x,y,z,w)=(rx,ry,r^2z,r^3w)$ have Jacobian $r^7$. The gauge $\rho$ is comparable with the left control distance from the identity and satisfies $\rho(g^{-1})\asymp\rho(g)$ and a quasi-triangle inequality. We use
\begin{equation}\label{eq2.2}
(f*k)(g)=\int_G f(v)k(v^{-1}g)\dd v,\qquad F(\calL)f=f*K_F.
\end{equation}
The kernel difference in the Calder\'on-Zygmund argument is $K(g_0^{-1}g)-K(g)$, with $g_0\in G$. Only the coordinate $w$ is central. Fourier transformation in the higher layers means the Euclidean partial Fourier transform in $(z,w)$.

Define the reduced operators by
\[
H_{\eta,\zeta,\lambda}=-\partial_x^2+\bigl(\eta+\zeta x+\lambda x^2/2\bigr)^2.
\]
The coefficients of $X$ and $Y$ depend only on $x$, so $\calL$ commutes with Euclidean translations in $(y,z,w)$. Thus $\FA$ conjugates $\calL$ to the direct integral of the operators $H_{\eta,\zeta,\lambda}$ acting in $x$, and $F(\calL)$ to the direct integral of $F(H_{\eta,\zeta,\lambda})$. Since $K_F=F(\calL)\delta_e$ in the sense of distributions, we obtain
\begin{equation}\label{eq2.3}
\FA K_F(x,\eta,\zeta,\lambda)=\bigl[F(H_{\eta,\zeta,\lambda})\delta_0\bigr](x).
\end{equation}
For smooth compactly supported $F$, both sides are smooth by the Schwartz calculus, so the identity is pointwise as well.

When $\lambda\ne0$, set $c=\zeta/\lambda$ and $\nu=(\zeta^2-2\eta\lambda)/2$. The map $(c,\nu,\lambda)\mapsto(\eta,\zeta,\lambda)$ has absolute Jacobian one. After translation,
\begin{equation}\label{eq2.4}
P_{\nu,\lambda}=-\partial_u^2+\bigl(\tfrac12\lambda u^2-\nu/\lambda\bigr)^2,
\qquad
\norm{K_F}_2^2=(2\pi)^{-3}\int_{\R^2}\norm{F(P_{\nu,\lambda})}_{\mathrm{HS}}^2\dd\nu\dd\lambda.
\end{equation}
The column in \eqref{eq2.3} becomes $k_{F(P)}(x+c,c)$. Integrating the squared modulus of this column in $x$ and $c$ gives $\norm{F(P_{\nu,\lambda})}_{\mathrm{HS}}^2$.

To estimate the moments in the higher layers, we localize in the corresponding Fourier variables. Choose $\vartheta\in C_c^\infty(\R^2\setminus\{0\})$ such that
\begin{equation}\label{eq2.5}
\sum_{j\in\mathbb Z}\vartheta(2^{-j}\zeta,2^{-2j}\lambda)=1
\quad\text{for }(\zeta,\lambda)\ne(0,0).
\end{equation}
The support of $\vartheta$ is contained in the annulus $\Omega_0=\{(a,b)\in\R^2:c_0\le |a|+|b|^{1/2}\le C_0\}$. For $h=2^j$ put
\[
\begin{gathered}
a=\zeta/h,\quad b=\lambda/h^2,\quad u=hx,\quad\theta=(\eta,a,b),\\
p_\theta(u)=\eta+au+\tfrac12bu^2,\qquad
P_{h,\theta}=-h^2\partial_u^2+p_\theta(u)^2,\qquad
v_F(h,\theta)=F(P_{h,\theta})\delta_0.
\end{gathered}
\]
Scaling the delta function, we obtain
\begin{equation}\label{eq2.6}
\FA K_F(x,\eta,ha,h^2b)=h\,v_F(h,\theta)(hx).
\end{equation}
Let $K_{F,h}$ be the kernel obtained by multiplying the left side of \eqref{eq2.3} by $\vartheta(\zeta/h,\lambda/h^2)$. Euclidean Plancherel yields
\begin{equation}\label{eq2.7}
\norm{z^kw^lK_{F,h}}_2^2=(2\pi)^{-3}h^{4-2k-4l}
\int_{\R^3}\norm{\partial_a^k\partial_b^l
(\vartheta v_F)(h,\theta)}_{L^2_u}^2\dd\theta.
\end{equation}
The factor $h^4$ is the Fourier Jacobian $h^3$ times $h^2$ from the prefactor in \eqref{eq2.6} and $h^{-1}$ from the change of variables $u=hx$.

Fix $\psi\in C_c^\infty((0,\infty))$ and put $\psi_\tau(E)=\psi(\tau^2E)$ for $1\le \tau\le2$.
\begin{lemma}\label{Lem2.1}
For every multi-index $\omega=(\omega_1,\omega_2)$, every $M,N\ge0$, and any finite number of $\tau$-derivatives,
\begin{equation}\label{eq2.8}
\norm{(1+P_{h,\theta})^M
\partial_a^{\omega_1}\partial_b^{\omega_2}
\psi_\tau(P_{h,\theta})\delta_0}_2
\le C h^{\omega_1+2\omega_2-1/2}(1+|\eta|+h|a|+h^2|b|)^{-N}.
\end{equation}
The same estimate holds after the specified $\tau$-derivatives, for all $h>0$.
\end{lemma}
\begin{proof}
Hulanicki's theorem defines a map from $\Sch(\R)$ to $\Sch(G)$ by $F\mapsto K_F$. See \cite{Hul} and \cite[Theorem 1.1]{MRT}. To check that the map has a closed graph, suppose that $F_n\to F$ in $\Sch(\R)$ and $K_{F_n}\to K$ in $\Sch(G)$. Then $F_n(\calL)\to F(\calL)$ in operator norm on $L^2(G)$ by the spectral theorem, and convolution with $K_{F_n}$ converges to convolution with $K$ by Young's inequality. It follows that $K=K_F$, so the closed graph theorem implies continuity. The kernels $K_{\psi_\tau}$ and their specified $\tau$-derivatives form a bounded set in $\Sch(G)$, whose partial Fourier transforms are bounded in $\Sch(\R^4)$. From \eqref{eq2.6},
\[
\partial_a^{\omega_1}\partial_b^{\omega_2}
\psi_\tau(P_{h,\theta})\delta_0
=h^{\omega_1+2\omega_2-1}(\partial_\zeta^{\omega_1}\partial_\lambda^{\omega_2}
\FA K_{\psi_\tau})(u/h,\eta,ha,h^2b).
\]
The substitution $u=hx$ proves \eqref{eq2.8} for $M=0$. For integer $M$, apply $(1+H_{\eta,\zeta,\lambda})^M$ to the differentiated Schwartz function before scaling. This is a differential operator with polynomial coefficients and preserves Schwartz bounds. For real $M\ge0$, use a larger integer power of $1+P_{h,\theta}$.
\end{proof}

Write $D=h\partial_u$ and $B=ip_\theta(u)$.
\begin{lemma}\label{Lem2.2}
For every word $A$ of length $m$ in $D,B$, every $R\ge1$, and all $h,\theta$,
\begin{equation}\label{eq2.9}
\norm{A\ind_{[0,R]}(P_{h,\theta})}_{2\to2}\le C_A R^{m/2}.
\end{equation}
In particular, $\norm{A(1+P_{h,\theta})^{-r}}_{2\to2}\le C_{A,r}$ whenever $r>m/2$.
\end{lemma}
\begin{proof}
The operators $D,B$, together with $[D,B]=ihp_\theta'$ and $[D,[D,B]]=ih^2b$, integrate to a unitary Engel representation. Explicitly, for $g=(x,y,z,w)$,
\[
(\pi(g)f)(u)=\exp\bigl(i[yp_\theta(u)+zhp_\theta'(u)+wh^2b]\bigr)\,f(u+hx).
\]
By \eqref{eq2.1}, $\pi(g)\pi(g')=\pi(gg')$, and $d\pi(X)=D$, $d\pi(Y)=B$. Write $\pi(k)=\int_G k(g)\pi(g^{-1})\dd g$, as in \cite[Proposition 1.5]{MRT}. Choose $\chi_1\in C_c^\infty(\R)$ equal to one on $[0,1]$. If $\mathcal A$ is the corresponding left-invariant differential operator, then $d\pi(\mathcal A)=A$. By Hulanicki's theorem, $\mathcal A K_{\chi_1}\in L^1(G)$. Homogeneity yields
\[
\norm{\mathcal A K_{\chi_1(\cdot/R)}}_1=R^{m/2}\norm{\mathcal A K_{\chi_1}}_1.
\]
By \cite[Proposition 1.5]{MRT}, $\pi(K_{\chi_1(\cdot/R)})=\chi_1(P/R)$. Integration by parts yields $\pi(\mathcal A K_{\chi_1(\cdot/R)})=A\chi_1(P/R)$. The norm of an integrated kernel is bounded by its $L^1$ norm, so applying these identities to
\[
A\ind_{[0,R]}(P)=A\chi_1(P/R)\ind_{[0,R]}(P)
\]
proves \eqref{eq2.9}. For the second assertion, split the spectrum of $P$ into $[0,1]$ and the intervals $(2^j,2^{j+1}]$, $j\ge0$. By \eqref{eq2.9}, the norm of $A(1+P)^{-r}$ restricted to these spectral subspaces is at most $C_A$ and $C_A2^{j(m/2-r)}$, respectively. The sum converges for $r>m/2$.
\end{proof}

\subsection{Averaged functional calculus}\label{Sec2.2}
Parameter derivatives of the functional calculus involve divided differences, whose averages over spectral dilations can be estimated in Sobolev norms. For $f_\tau(E)=F(\tau^2E)$ let $f_\tau^{[k]}$ denote the divided difference of order $k$, extended continuously to repeated nodes (the points $E_0,\ldots,E_k$ are called nodes).
\begin{lemma}\label{Lem2.3}
For $J\Subset(0,\infty)$, $k\in\N_0$ and arbitrary $E_0,\ldots,E_k\ge0$,
\begin{equation}\label{eq2.10}
\int_1^2 |f_\tau^{[k]}(E_0,\ldots,E_k)|^2\dd \tau
\le C_{J,k}\norm{F}_{\Hs{k}}^2,\qquad F\in C_c^\infty(J).
\end{equation}
For every compact interval $I\Subset(0,\infty)$, every $E,E'\in I$, $j\in\N_0$, $0<\kappa<1$ and $F\in C_c^\infty(J)$, one also has
\begin{equation}\label{eq2.11}
\norm{f_\tau^{(j)}(E')-f_\tau^{(j)}(E)}_{L^2_\tau(1,2)}
\le C_{I,j,\kappa}|E'-E|^\kappa\norm{F}_{\Hs{j+\kappa}}.
\end{equation}
\end{lemma}
\begin{proof}
The function $f_\tau^{(k)}(v)=\tau^{2k}F^{(k)}(\tau^2v)$ vanishes unless $v$ belongs to a compact subinterval of $(0,\infty)$ that does not depend on $\tau$. On that interval, $r=\tau^2v$ has a uniformly nonzero derivative in $\tau$. Thus $\norm{f_\tau^{(k)}(v)}_{L^2_\tau}\le C\norm{F^{(k)}}_2$ uniformly for $v\ge0$. Let $\Delta_k=\{(t_0,\ldots,t_k):t_i\ge0,\ \sum_i t_i=1\}$, equipped with the coordinate measure $\dd t_1\cdots\dd t_k$, where $t_0=1-\sum_{i=1}^kt_i$. This measure has total mass $1/k!$. For $k=0$, it is the unit point mass. The Hermite-Genocchi formula
\[
f_\tau^{[k]}(E_0,\ldots,E_k)=\int_{\Delta_k}f_\tau^{(k)}\Bigl(\sum_{i=0}^k t_iE_i\Bigr)\dd t
\]
and Minkowski's inequality prove \eqref{eq2.10}. For \eqref{eq2.11}, change variables $r=\tau^2E$. The difference is a bounded smooth factor times $F^{(j)}((E'/E)r)-F^{(j)}(r)$ on a compact subinterval of $(0,\infty)$. The associated difference operator has bounded $L^2\to L^2$ norm and $H^1\to L^2$ norm $O(|E'-E|)$ by the fundamental theorem of calculus. Interpolation proves the claim.
\end{proof}

To apply these scalar estimates to vectors depending on $\tau$, we estimate their spectral coefficients. For a Hilbert space $\mathcal H$, write $L^2_\tau\mathcal H=L^2((1,2);\mathcal H)$ and $H^1_\tau\mathcal H=H^1((1,2);\mathcal H)$. When $\mathcal H=L^2(\R_u)$, we use $L^2_\tau L^2_u$ and $H^1_\tau L^2_u$. Applying the scalar one-dimensional Sobolev inequality to each coefficient of $\mathbf v\in H^1_\tau\mathcal H$ in an orthonormal basis $(\varphi_n)$ yields
\begin{equation}\label{eq2.12}
\sum_n\sup_{1\le \tau\le2}|\ip{\mathbf v(\tau)}{\varphi_n}|^2
\le C\int_1^2\bigl(\norm{\mathbf v(\tau)}^2+\norm{\mathbf v'(\tau)}^2\bigr)\dd \tau.
\end{equation}
In our applications, $\mathbf v(\tau)$ is $\psi_\tau(H)\delta_0$ or one of its parameter derivatives, where $\psi$ is a smooth cutoff. Their smooth dependence on $\tau$ supplies the regularity required in \eqref{eq2.12}, independently of $F$.

Put $\Lambda(h)=1+\log(1/h)$ for $0<h\le1$, and write $\ad_H(A)=[H,A]$, with $\ad_H^0(A)=A$. Let $H=H_\theta\ge0$ be a family of self-adjoint operators with compact resolvent on a fixed Hilbert space $\mathcal H$, depending on finitely many real parameters $\theta$, and let $0<h\le1$. Fix a parameter $\theta_0$ in the interior of the parameter set. For $m\in\N_0$ put
\[
\mathcal H_m=D((1+H_{\theta_0})^m),\qquad
\norm v_{\mathcal H_m}=\norm{(1+H_{\theta_0})^mv}_{\mathcal H}.
\]
Fix $J\Subset(0,\infty)$ and an interval $I\Subset(0,\infty)$ containing a neighborhood of $J_0=\bigcup_{1\le \tau\le2}\tau^{-2}J$.

\paragraph*{Hypothesis (H).} Conditions (H1) and (H2) hold on a neighborhood of $\theta_0$. Conditions (H3) and (H4) are imposed at $\theta_0$, with $H=H_{\theta_0}$ and all parameter derivatives evaluated there. In (H3) and (H4) we assume $0<h\le1$; a variant for $1\le h\le h_0$ is recorded at the end of Lemma~\ref{Lem2.4}.
\begin{enumerate}[label=(H\arabic*),leftmargin=3em]
\item There is a dense subspace of $\mathcal H$, invariant under every $H_\theta$, which is a core for every power of $1+H_\theta$. For each $m$, the norms $\norm{(1+H_\theta)^mv}$ with $\theta$ near $\theta_0$ are equivalent to $\norm v_{\mathcal H_m}$, with constants uniform on a neighborhood of $\theta_0$.
\item For every $m\in\N_0$ there is $r_m$ such that $\theta\mapsto H_\theta$ is a $C^2$ map with values in the bounded operators from $\mathcal H_{m+r_m}$ to $\mathcal H_m$.
\item For some $M_0$, $\Tr(1+H)^{-M_0}\le C_*h^{-1}$, and every interval of length $h$ in $I$ contains at most $C_*\Lambda(h)$ eigenvalues of $H$, counted with multiplicity.
\item The symmetric operators $A_i=\partial_iH$ and $B_{ij}=\partial_i\partial_jH$ satisfy, for a fixed $r_A$, every $m,N\in\N_0$, and every $A\in\{A_i,B_{ij}\}$,
\begin{align}
\norm{(1+H)^mA(1+H)^{-m-r_A}}&\le C_m,\label{eq2.13}\\
\norm{\ad_H^N(A)\ind_{[0,R]}(H)}&\le C_{N,R}h^N
\quad(R\ge1).\label{eq2.14}
\end{align}
\end{enumerate}
All operator norms in (H) are taken on $\mathcal H$.

\begin{lemma}\label{Lem2.4}
Assume (H), and evaluate the operators and parameter derivatives at $\theta_0$. For $k=0,1,2$, every parameter derivative $\partial^k$ of total order $k$, and $\mathbf v\in H^1_\tau\mathcal H$,
\begin{equation}\label{eq2.15}
\norm{(\partial^kf_\tau(H))\mathbf v(\tau)}_{L^2_\tau\mathcal H}
\le C\Lambda(h)^k\norm{F}_{\Hs{k}}
\norm{\mathbf v}_{H^1_\tau\mathcal H}.
\end{equation}
For fixed structural indices $M_0$ and $r_A$ in (H), the constant $C$ is bounded by a polynomial, depending only on these indices, $k$, $J$ and $I$, in finitely many of the quantitative constants in (H3) and (H4). In particular, if these constants grow polynomially in an auxiliary parameter, so does $C$. If the bound $C_*\Lambda(h)$ in (H3) is replaced by $C_*$, then $\Lambda(h)$ may be replaced by $1$ in \eqref{eq2.15}. The same conclusion with this replacement holds for $1\le h\le h_0$, provided (H1) and (H2) hold for each $h$ in that interval, and the trace bound and the bounds in (H4) are uniform there; the constant may then depend on $h_0$.
\end{lemma}
\begin{proof}
By (H1) and (H2), the resolvent identity yields the first and second parameter derivatives of the resolvent, after replacing the domain by $D((1+H)^{m+r})$ with a larger $r$ for each factor $A_i$ or $B_{ij}$. At the parameter being evaluated, choose an eigenbasis, $H\varphi_n=E_n\varphi_n$, and write $A_{nm}=\ip{A\varphi_m}{\varphi_n}$. On a bounded energy interval, \eqref{eq2.14} and the $N=0$ bound give
\begin{equation}\label{eq2.16}
|A_{nm}|\le C_N\min\{1,(h/|E_n-E_m|)^N\}.
\end{equation}
Use the first bound for coincident nodes. Divide $I$ into intervals of length $h$ and take $N=2$. The matrix with entries $|A_{nm}|$, restricted to $E_n,E_m\in I$, then has row and column sums at most $C\Lambda(h)$. End intervals shorter than $h$ are absorbed into neighboring intervals. For $h$ bounded below, the trace bound controls the total number of eigenvalues in $I$, since this number is at most $(1+\sup I)^{M_0}\Tr(1+H)^{-M_0}$. The same argument applies on $1\le h\le h_0$ under the stated uniform bounds.

For an almost analytic extension $\widetilde f_\tau$ of sufficiently high order, the resolvent formula is
\[
f_\tau(H)=\frac1\pi\int_{\C}\partial_{\bar z}\widetilde f_\tau(z)(H-z)^{-1}\dd A(z),
\]
where $dA$ is Lebesgue measure on $\C$ and $\partial_{\bar z}=\tfrac12(\partial_{\Re z}+i\partial_{\Im z})$. See \cite[Chapter 8]{DS}. Differentiation between the stated graph spaces yields
\begin{align}
(\partial_i f_\tau(H))_{nm}
&=f_\tau^{[1]}(E_n,E_m)(A_i)_{nm},\label{eq2.17}\\
(\partial_i\partial_j f_\tau(H))_{nm}
&=f_\tau^{[1]}(E_n,E_m)(B_{ij})_{nm}\notag\\
&\quad+\sum_\ell f_\tau^{[2]}(E_n,E_\ell,E_m)
\bigl((A_i)_{n\ell}(A_j)_{\ell m}+(A_j)_{n\ell}(A_i)_{\ell m}\bigr).
\label{eq2.18}
\end{align}
These are initially identities on the core in (H1). The estimates below justify all matrix sums for $\mathbf v\in H^1_\tau\mathcal H$.

Write $\mathbf v_m(\tau)=\ip{\mathbf v(\tau)}{\varphi_m}$ and set $c_m=\sup_\tau|\mathbf v_m(\tau)|$. By \eqref{eq2.12}, $\sum_mc_m^2\le C\norm{\mathbf v}_{H^1_\tau\mathcal H}^2$. For the first derivative with both indices in $I$, weighted Cauchy-Schwarz in each row and \eqref{eq2.10} give
\[
\int_1^2\sum_n\Big|\sum_m A_{nm}f_\tau^{[1]}(E_n,E_m)\mathbf v_m(\tau)\Big|^2\dd \tau
\le C\Lambda(h)^2\norm{F}_{\Hs1}^2\sum_mc_m^2.
\]
For the second derivative, consider the index triples $(n,\ell,m)$ in \eqref{eq2.18}. The absolute weights $|(A_i)_{n\ell}(A_j)_{\ell m}|$ have total row and column sums at most $C\Lambda(h)^2$. For the second derivative we obtain $C\Lambda(h)^4\norm{F}_{\Hs2}^2\sum_mc_m^2$. The term containing $B_{ij}$ is estimated by the argument for the first derivative, using $\norm F_{H^1}\le C\norm F_{H^2}$. For $k=0$, apply \eqref{eq2.10} separately to each spectral coefficient and sum using \eqref{eq2.12}.

We now allow nodes outside $I$. A divided difference in \eqref{eq2.17} or \eqref{eq2.18} vanishes unless one node lies in $J_0$. This follows from its formula in function values, including the limits at repeated nodes. If another node lies outside $I$, then in a pair $(n,m)$ or a triple $(n,\ell,m)$ two consecutive indices have energies that differ by at least a positive constant. On a bounded energy region the corresponding factor is $O(h^N)$ for every $N$ by \eqref{eq2.16}.

For unbounded indices, the trace and graph bounds imply, for some $q$,
\begin{equation}\label{eq2.19}
\sum_n|A_{n\ell}|\le Ch^{-1/2}(1+E_\ell)^q.
\end{equation}
To see this, apply Cauchy-Schwarz after inserting $(1+E_n)^{-m}$, with $2m\ge M_0$, and then use \eqref{eq2.13}. If $E_n\le R$ and $E_\ell\ge2R$, symmetry and \eqref{eq2.14} give
\begin{equation}\label{eq2.20}
|A_{n\ell}|\le C_{N,R}h^N(1+E_\ell)^{-N}.
\end{equation}
For example, if $E_n\in J_0$ and $E_\ell\ge2R$, summing over the remaining index by \eqref{eq2.19}, and then summing $n,\ell$ with the trace bound, we obtain
\[
\sum_{E_n\in J_0,\ E_\ell\ge2R}\sum_m|(A_i)_{n\ell}(A_j)_{\ell m}|
\le C_Nh^{N-5/2},
\]
after increasing $N$ to dominate $q$. If the middle node is bounded and the final node unbounded, use \eqref{eq2.20} for the last factor. The sums over bounded nodes contribute at most $O(h^{-2})$. If the node in $J_0$ is the middle node, sum the two adjacent factors separately, using \eqref{eq2.20} for an unbounded endpoint. The case with the final node in $J_0$ follows by taking adjoints. If all nodes are bounded, there are $O(h^{-3})$ triples, and a factor whose two indices are separated contributes $O(h^{N-3})$ in total.

The index triples with a node outside $I$ have total absolute coefficient sum $O(h^L)$ for every $L$. Minkowski's inequality, \eqref{eq2.10}, and $c_m\le(\sum_mc_m^2)^{1/2}$ bound their contribution by the right side of \eqref{eq2.15}. For terms with a single factor $A_i$, the same alternatives involve two nodes. If one node lies in $J_0$ and the other outside $I$, \eqref{eq2.16} and \eqref{eq2.20}, followed by the trace bound, give a total sum $O(h^{N-2})$. This also controls the single $B_{ij}$ term. Each constant above is a finite sum of products of the constants in (H3) and (H4). The asserted polynomial dependence follows.
\end{proof}

Fix a smooth cutoff $\psi$ supported in a compact subinterval of $(0,\infty)$ and write $\psi_\tau(E)=\psi(\tau^2E)$.
\begin{corollary}\label{Cor2.5}
Assume (H), and evaluate the operators and parameter derivatives at $\theta_0$. For every integer $M\ge0$, every $k\in\{0,1,2\}$, and $1\le \tau\le2$,
\[
\norm{(1+H)^M(\partial^k\psi_\tau(H))(1+H)^M}_{2\to2}
\le C_{M,\psi}\Lambda(h)^k.
\]
The estimate holds after any finite number of $\tau$-derivatives. If each $h$-interval contains only a uniformly bounded number of eigenvalues, replace $\Lambda(h)$ by $1$.
\end{corollary}
\begin{proof}
Use \eqref{eq2.17}-\eqref{eq2.18} with the smooth function $\psi_\tau$. Its divided differences, and their specified $\tau$-derivatives, are uniformly bounded. For $E_n,E_m\in I$, the factors $(1+E_n)^M(1+E_m)^M$ are uniformly bounded. Outside $I$, the weighted version of \eqref{eq2.19} is
\[
\sum_n(1+E_n)^M|A_{n\ell}|
\le C_Mh^{-1/2}(1+E_\ell)^{q_M}.
\]
The arbitrarily high energy decay in \eqref{eq2.20} absorbs $q_M$ and the endpoint weights. The absolute sum over index triples with a node outside $I$ is again $O(h^L)$. Schur's test proves the claim. These absolute bounds allow us to pass to the limit in the finite weighted matrix sums, and the limits are the weighted derivatives computed between graph spaces.
\end{proof}

To apply this calculus to the kernel column $f_\tau(H)\delta_0$, choose $\psi=1$ on a neighborhood of $J$, supported in a larger compact subinterval of $(0,\infty)$, and use
\begin{equation}\label{eq2.21}
f_\tau(H)\delta_0=f_\tau(H)\psi_\tau(H)\delta_0.
\end{equation}
The lemmas are then applied with that larger interval. For two parameter derivatives, the vectors needed are $\psi_\tau(H)\delta_0$, $\partial_i\psi_\tau(H)\delta_0$ and $\partial_i\partial_j\psi_\tau(H)\delta_0$, in $H^1_\tau\mathcal H$.

\section{Uniform spectral estimates}\label{Sec3}
\subsection{Quartic spectral clusters}\label{Sec3.1}
The trace, cluster and commutator estimates for the quartic family have constants with polynomial losses in $|b|^{-1}$ and $|\eta|$. We use these estimates only where $|b|^{-1}\lesssim\br\eta$. Let $p(u)=\eta+au+bu^2/2$, $b\ne0$, and $P=-h^2\partial_u^2+p^2$.
\begin{lemma}\label{Lem3.1}
For $0<h\le1$, $M>3/4$ and $I\Subset(0,\infty)$,
\begin{align}
\Tr(1+P)^{-M}&\le C_Mh^{-1}|b|^{-1/2},\label{eq3.1}\\
\#\bigl(\Spec(P)\cap[E,E+h]\bigr)
&\le C_I(1+|b|^{-1/2})\Lambda(h),\label{eq3.2}
\end{align}
whenever $E,E+h\in I$. The constants are independent of $a,\eta$. The trace estimate \eqref{eq3.1} remains valid for all $h>0$.
\end{lemma}
\begin{proof}
The heat trace inequality reads
\[
\Tr e^{-tP}\le(4\pi h^2t)^{-1/2}\int_\R e^{-tp(u)^2}\dd u.
\]
It follows from the Trotter product formula and the Golden-Thompson inequality; see \cite{SimFI}. One first proves it for the Dirichlet realization on a finite interval and then lets the interval increase to $\R$. The Dirichlet traces increase to the trace on the line. Complete the square and rescale. Uniformly for $c\in\R$,
\[
\int_\R e^{-(v^2+c)^2}\dd v=\int_0^\infty e^{-(s+c)^2}s^{-1/2}\dd s
\le 2+\sqrt\pi.
\]
The resulting bound is $\Tr e^{-tP}\le Ch^{-1}|b|^{-1/2}t^{-3/4}$. Integrating against $t^{M-1}e^{-t}/\Gamma(M)$ proves \eqref{eq3.1}.

Put $U=p^2$ and $\Theta(E)=\int_\R\sqrt{(E-U(u))_+}\dd u$. For an energy $E$ we call $\{U<E-h\}$ the allowed region, $\{U>E+h\}$ the forbidden region, and $\{|U-E|\le h\}$ the transition region. On each of the at most two monotonic branches of $u\mapsto p(u)$, the inverse derivative is
\[
C|b|^{-1/2}|v-c|^{-1/2}\ind_{I_c}(v),\qquad c=\eta-a^2/(2b),
\]
where $I_c$ is a half-line. Since $E$ remains in a positive compact interval, $|v^2-E|\le h$ is contained in two intervals of length $O_I(h)$ for small $h$. The integral of $|v-c|^{-1/2}$ on any interval of length $h$ is $O(\sqrt h)$, uniformly in its location. Consequently,
\begin{equation}\label{eq3.3}
|\{u:|U(u)-E|\le h\}|\le C_I|b|^{-1/2}h^{1/2}.
\end{equation}
Moreover,
\begin{equation}\label{eq3.4}
0\le\Theta(E+h)-\Theta(E)\le C_I|b|^{-1/2}h\Lambda(h).
\end{equation}
On $|v^2-E|\le2h$ the difference of the square roots is $O(\sqrt h)$, so \eqref{eq3.3} suffices. On the remaining allowed region it is at most $Ch(E-v^2)^{-1/2}$. Decompose into dyadic distances $d\asymp2^jh$ from the endpoints $\pm\sqrt E$. On a piece of length $O(d)$, the second factor is $O_I(d^{-1/2})$, while the integral of $|v-c|^{-1/2}$ is $O(\sqrt d)$. Each piece contributes $O_I(h|b|^{-1/2})$, and there are $O_I(\Lambda(h))$ pieces.

For single-well potentials, a Bohr-Sommerfeld approximation with logarithmic error is proved in \cite[Proposition 2.6]{DMII}. For $U$ with up to two wells, the following zero count is uniform in $a$, $\eta$ and $b$. On a component of $U<E-h$, set $k(u)=h^{-1}\sqrt{E-U(u)}$. Write a nonzero real solution as $v=r\sin\varphi$, $v'=kr\cos\varphi$, with a continuous angle. Differentiation gives
\[
\varphi'=k+\frac{k'}{2k}\sin(2\varphi).
\]
At a zero of $v$, $\varphi'=k>0$, so the zero count on any such interval differs from the angle increment divided by $\pi$ by at most two. The angle increment differs from $\int k$ by at most
\[
\frac12\int |k'/k|\dd u.
\]
The potential $U$ has at most four monotonic pieces, so the numbers of allowed and transition components are uniformly bounded. Split the allowed region into these monotonic pieces. On a monotonic piece with endpoints $a_1,a_2$,
\[
\int_{a_1}^{a_2}|k'/k|\,du
=\frac12\left|\log\frac{E-U(a_2)}{E-U(a_1)}\right|
\le\frac12\log(E/h).
\]
Here endpoints are interpreted by one-sided limits. Summing over the bounded number of pieces, the last integral is $O_I(\Lambda(h))$. On each component $C$ of $|U-E|\le h$, Sturm comparison bounds the number of zeros by $1+Ch^{-1/2}|C|$. By \eqref{eq3.3}, all these components contribute $O_I(1+|b|^{-1/2})$. A component of $U>E+h$ contains at most one zero, because two zeros would contradict the positive Dirichlet energy identity between them. The action omitted in the transition regions, divided by $h$, is $O_I(|b|^{-1/2})$.

Choose a finite interval whose endpoints lie outside $\{U\le E+h\}$. Apply Sturm oscillation to the solution vanishing at its left endpoint. This splitting introduces only a bounded number of endpoints, since $U$ has at most four monotonic pieces. Zeros at a splitting point change the total by at most that number. The two exterior forbidden intervals contribute at most two further zeros, independently of their lengths. The estimate remains uniform as the Dirichlet interval increases to $\R$.

The Dirichlet eigenvalues decrease to the eigenvalues on the line by the min-max principle: compactly supported smooth functions approximate every finite collection of eigenfunctions in the form norm, and domain inclusion supplies the reverse inequality. The counts on finite intervals converge at energies outside the spectrum on the line. At an eigenvalue, the two one-sided counts differ by one, since a confining one-dimensional Schr\"odinger operator has simple spectrum. The same estimate holds with either endpoint convention. All spatial splittings were made at $U=E\pm h$, so they remain valid when $E$ equals a critical value of $U$. The result is
\begin{equation}\label{eq3.5}
\#\{E_n\le E\}=\frac{\Theta(E)}{\pi h}+O_I\bigl((1+|b|^{-1/2})\Lambda(h)\bigr).
\end{equation}
For $h$ bounded below, the same conclusion follows from the trace estimate with different constants. Subtract \eqref{eq3.5} at $E+h$ and $E$, and use \eqref{eq3.4}. Endpoint conventions change the count only by a bounded number. Thus \eqref{eq3.2} also holds at the top of the barrier.
\end{proof}

\begin{lemma}\label{Lem3.2}
Fix $\eta\in\R$ and write $\theta=(a,b)$. If $(a,b)$ ranges in a bounded set and $b\ne0$, the family $P_{h,\theta}$ satisfies Hypothesis (H) of Section~\ref{Sec2} for $0<h\le1$. Put $L=1+|\eta|+|b|^{-1}$. The quantitative constants in (H3) and (H4) are bounded by $CL^C$. Condition (H3) follows from Lemma~\ref{Lem3.1}; the first and second derivatives with respect to $a,b$ satisfy \eqref{eq2.13}-\eqref{eq2.14}, with $r_A=3$. For $h\ge1$, each fixed finite collection of estimates \eqref{eq2.13}-\eqref{eq2.14} has constants bounded by $C(1+h)^CL^C$. The common graph domains and local norm equivalence in (H1), and the $C^2$ property in (H2), hold for every fixed $h>0$.
\end{lemma}
\begin{proof}
The derivatives are
\begin{equation}\label{eq3.6}
\partial_aP=2up,\quad \partial_bP=u^2p,\quad
\partial_a\partial_bP=u^3,\quad
\partial_a^2P=2u^2,\quad\partial_b^2P=u^4/2.
\end{equation}
With $D=h\partial_u$, the normal-ordered commutator is
\begin{equation}\label{eq3.7}
[P,fD^j]=-2hf'D^{j+1}-h^2f''D^j-f\sum_{r=1}^j\binom jr h^r(p^2)^{(r)}D^{j-r}.
\end{equation}
For bounded $a,b$ and $b\ne0$, Young's inequality gives
\[
|u|^2\le 4|b|^{-1}|p(u)|+4|b|^{-1}|\eta|+4a^2|b|^{-2}.
\]
Thus every integer $n\ge0$ satisfies
\[
\|u^nv\|_2\le C_n L^{C_n}\bigl(\|v\|_2+\|p^{\lceil n/2\rceil}v\|_2\bigr),
\qquad L=1+|\eta|+|b|^{-1}.
\]
Consider monomials in normal order $u^\alpha p^\beta(p')^\gamma D^\delta$, and assign their degree to be $\operatorname{deg}=\alpha+\beta+\gamma+\delta$. Their scalar coefficients are kept separate from this degree. Expanding $p'=a+bu$ and applying the preceding inequality to $p^\beta D^\delta v$ reduces each monomial to words in $D,p$ of length at most $\operatorname{deg}$: for $n=\alpha+k$, $0\le k\le\gamma$,
\[
\lceil n/2\rceil+\beta+\delta
\le n+\beta+\delta\le \operatorname{deg}\quad(n\ge1),
\]
and the case $n=0$ is immediate. Lemma~\ref{Lem2.2} now bounds the spectral restriction by
\[
\|u^\alpha p^\beta(p')^\gamma D^\delta\ind_{[0,R]}(P)\|
\le C_{\alpha,\beta,\gamma,\delta}L^C R^{\operatorname{deg}/2},\qquad R\ge1.
\]

Differentiation in $u$ does not increase the degree of a polynomial in $u,p,p'$: the derivatives of these three generators are $1,p',b$. The derivatives $(p^2)^{(r)}$, $r=1,\ldots,4$, are $2pp'$, $2(p')^2+2bp$, $6bp'$ and $6b^2$, of degrees $2$, $2$, $1$ and $0$. In \eqref{eq3.7}, the first term has degree at most $\operatorname{deg}+1$, the second at most $\operatorname{deg}$, and a term in the sum has degree at most $\operatorname{deg}+2-r\le \operatorname{deg}+1$. Every term also has a factor $h$. Induction shows that
\[
\ad_P^j(A)=h^j\sum_\mu c_{j,\mu}(h,a,b)
u^{\alpha_\mu}p^{\beta_\mu}(p')^{\gamma_\mu}D^{\delta_\mu},
\qquad\alpha_\mu+\beta_\mu+\gamma_\mu+\delta_\mu\le \operatorname{deg}(A)+j,
\]
where $\operatorname{deg}(A)\le4$ for all the operators in \eqref{eq3.6}. For $0<h\le1$, the coefficients are bounded by constants depending on $j$ and the bounded coefficient set. The resulting estimate is
\[
\|\ad_P^j(A)\ind_{[0,R]}(P)\|\le C_jL^{C_j}h^j R^{(\operatorname{deg}(A)+j)/2}.
\]
This gives \eqref{eq2.14} with $C_{N,R}=C_NL^{C_N}R^{(\operatorname{deg}(A)+N)/2}$. The dyadic summation in Lemma~\ref{Lem2.2} also implies
\[
\|\ad_P^j(A)(1+P)^{-j-3}\|\le C_jL^{C_j},
\]
since $(\operatorname{deg}(A)+j)/2-j-3\le-1-j/2$. On the common core,
\[
(1+P)^mA(1+P)^{-m-3}=\sum_{j=0}^m\binom mj\ad_P^j(A)(1+P)^{-j-3}.
\]
The bounded extensions of the terms on the right prove \eqref{eq2.13} with the shift $r_A=3$, independently of $m$. When $h\ge1$, the same finite expansions retain only finitely many additional powers of $h$, so their constants have polynomial growth in $h,L$ as asserted.

We now check Hypothesis (H) of Section~\ref{Sec2} for $H_\theta=P_{h,\theta}$, with $h>0$ fixed, $\theta=(a,b)$ in a bounded set and $b\ne0$. Condition (H3) is Lemma~\ref{Lem3.1}, and (H4) is \eqref{eq2.13}-\eqref{eq2.14}. Only these two conditions enter the constants in Lemma~\ref{Lem2.4}. The qualitative conditions (H1) and (H2) justify differentiation of the resolvent.

For (H1), note first that, for $v\in\Sch(\R)$,
\[
\|Pv\|_2^2=\|h^2v''\|_2^2+\|p^2v\|_2^2+2h^2\|pv'\|_2^2-2h^2\int_\R(p'^2+pp'')|v|^2\dd u.
\]
Since $b\ne0$, $|p'^2+pp''|\le\delta p^4+C_\delta$ for every $\delta>0$, locally uniformly in $\theta$. For $s\in\N_0$, put $\mathcal N_s(v)=\sum_{r+2j\le s}\|u^r\partial_u^jv\|_2$. After retaining the mixed positive term and absorbing the lower order terms, the identity implies $\mathcal N_4(v)\le C\|(1+P)v\|_2$. Apply this estimate to $u^r\partial_u^jv$ with $r+2j\le s$. The commutator with $P$ contains monomials of weight at most $r+2j+1$, so the normal-ordered estimate is
\[
\mathcal N_{s+4}(v)\le C_s\bigl(\mathcal N_s((1+P)v)+\mathcal N_{s+1}(v)\bigr).
\]
The last term is absorbed using
\[
\mathcal N_{s+1}(v)\le\delta\mathcal N_{s+4}(v)+C_{s,\delta}\mathcal N_s(v).
\]
Indeed, for a highest order term with $r\ge1$, use $|u|^r\le\delta|u|^{r+3}+C_\delta|u|^{r-1}$. For a pure derivative, use one-dimensional Fourier interpolation. Choosing $\delta$ small and then taking $s=4m-4$ yields by induction
\[
c_m\,\mathcal N_{4m}(v)\le\|(1+P_\theta)^mv\|_2+\|v\|_2\le C_m\,\mathcal N_{4m}(v),
\]
with constants locally uniform in $\theta$. For the domain assertion, suppose first that $v,Pv\in L^2$ distributionally. Local elliptic regularity permits integration by parts against $\chi_R^2v$, where $\chi_R(u)=\chi(u/R)$ is a real cutoff. The resulting form identity bounds $h(\chi_Rv)'$ and $p\chi_Rv$ uniformly; its cutoff error is $h^2\|\chi_R'v\|_2^2$. Passing to the limit places $v$ in the form domain and then in the operator domain. Apply the base estimate to $\chi_Rv$: the commutator terms $-2h^2\chi_R'v'-h^2\chi_R''v$ tend to zero in $L^2$, so the estimate extends to this maximal domain. For higher powers, induct on $m\ge2$ and set $s=4m-4$. Using the displayed commutator bound, proceed in two steps:
\begin{enumerate}[label=(\alph*),leftmargin=3em]
\item Apply the maximal-domain estimate to words of weight at most $s-1$, using the induction hypothesis to control their commutators. This yields $\mathcal N_{s+3}(v)<\infty$.
\item Apply the same estimate to words of weight at most $s$; the extra weight required by their commutators is controlled by the preceding step. This yields $\mathcal N_{s+4}(v)<\infty$ and extends the recurrence to the next operator power.
\end{enumerate}
Conversely, Leibniz's rule identifies each such power on the weighted Sobolev space. By cutting off and then mollifying, we approximate every function in the weighted Sobolev space in the $\mathcal N_{4m}$ norm. Thus both $C_c^\infty(\R)$ and $\Sch(\R)$ are cores for every power. Since $\mathcal N_{4m}$ does not depend on $\theta$, the graph norms of nearby parameters are equivalent with locally uniform constants, and $\Sch(\R)$ is a dense invariant subspace which is a core for every power of $1+P_\theta$.

For (H2), the potential $p_\theta^2$ is a polynomial of degree two in $\theta$, so by \eqref{eq3.6}
\[
P_\theta=P_{\theta_0}+\sum_i(\theta-\theta_0)_iA_i+\tfrac12\sum_{i,j}(\theta-\theta_0)_i(\theta-\theta_0)_jB_{ij},
\]
where $A_i=\partial_iP_{\theta_0}$ and the $B_{ij}$ are independent of $\theta$. By \eqref{eq2.13} at $\theta_0$, $A_i$ and $B_{ij}$ are bounded from $\mathcal H_{m+3}$ to $\mathcal H_m$. The map $\theta\mapsto P_\theta$ is polynomial with values in the bounded operators from $\mathcal H_{m+3}$ to $\mathcal H_m$, as required in (H2) with $r_m=3$.
\end{proof}

\subsection{A uniformly confining comparison family}\label{Sec3.2}
The region in which the second well moves to infinity requires estimates uniform as $b$ tends to zero. We obtain them from a comparison operator with one uniformly confining well. In this subsection and the next section, $\xi$ denotes the one-dimensional spatial variable. The comparison operators act on $L^2(\R,d\xi)$. Fix $\sigma\in\{-1,1\}$. Choose an even $\chi_0\in C_c^\infty(\R)$, equal to one on $[-2,2]$ and supported in $[-4,4]$. For a sufficiently small $d_0>0$, define
\[
\Vt_{\sigma,t}(\xi)=\sigma \xi+\tfrac t2 \xi^2\chi_0(t\xi/d_0),
\qquad U_{\sigma,t}=\Vt_{\sigma,t}^2,
\qquad \Ht_{\ee,t}=-\ee^2\partial_\xi^2+U_{\sigma,t}(\xi).
\]
Choose $d_0$ so small that both $|r\chi_0(r/d_0)|$ and $|r\chi_0(r/d_0)+r^2\chi_0'(r/d_0)/(2d_0)|$ are less than $1/4$. Direct differentiation gives
\begin{align}
c\xi^2\le U_{\sigma,t}(\xi)&\le C\xi^2,
&c\xi^2\le \xi U_{\sigma,t}'(\xi)&\le C\xi^2,\label{eq3.8}\\
|\partial_t^j\partial_\xi^kU_{\sigma,t}(\xi)|
&\le C_{j,k}(1+|\xi|)^{2+j-k}.
&&\label{eq3.9}
\end{align}
These bounds hold uniformly for bounded $t$. In the transition region, \eqref{eq3.9} follows from $|t\xi|\asymp d_0$. Elsewhere it follows from the polynomial formula.

\Needspace{8\baselineskip}
\begin{lemma}\label{Lem3.3}
Uniformly for $|t|\le t_0$ and $0<\ee\le\ee_0$, the following assertions hold.
\begin{enumerate}[label=(\roman*),leftmargin=3em]
\item The simple eigenvalues $e_n$ of $\Ht_{\ee,t}$, indexed by $n\ge0$, satisfy
\begin{equation}\label{eq3.10}
c\ee(n+1)\le e_n\le C\ee(n+1),\qquad e_{n+1}-e_n\ge c\ee.
\end{equation}
\item Write $\Ht=\Ht_{\ee,t}$. On the common domain of $(1+\Ht)^m$,
\begin{equation}\label{eq3.11}
\norm{(1+\Ht)^mv}_2\asymp_m\sum_{r+j\le2m}\norm{\xi^r(\ee\partial_\xi)^jv}_2,
\qquad m\in\N_0.
\end{equation}
The constants are independent of $t,\ee$, and $\Sch(\R)$ is a core for every such power.
\item Let $r,j,\ell\in\N_0$ and let $T$ denote multiplication by $\xi^r\partial_t^j\partial_\xi^\ell U_{\sigma,t}$. Put $s_T=\lceil(r+2+j)/2\rceil$. For every $m,N\in\N_0$ and $R\ge1$,
\[
\|(1+\Ht_{\ee,t})^mT(1+\Ht_{\ee,t})^{-m-s_T}\|\le C_{m,r,j,\ell},
\]
\[
\|\ad_{\Ht_{\ee,t}}^N(T)\ind_{[0,R]}(\Ht_{\ee,t})\|
\le C_{N,r,j,\ell}\ee^N (1+R)^{\lceil(r+2+j+N)/2\rceil}.
\]
The same bounds hold for multiplication by $\xi^r$, with the displayed shifts and exponents taken at $j=0$. For every fixed finite range of $r,j,\ell,m,N$, the constants depend on only finitely many of the derivative bounds in \eqref{eq3.9} and are uniform in $t,\ee$.
\item Fix $\ee>0$ and a reference parameter $t_*\in[-t_0,t_0]$, and set $\mathcal H_m=D((1+\Ht_{\ee,t_*})^m)$. For every $m\in\N_0$, the map
\[
t\longmapsto\Ht_{\ee,t}\quad\hbox{belongs to}\quad
C^2\bigl([-t_0,t_0];\mathcal L(\mathcal H_{m+3},\mathcal H_m)\bigr).
\]
Its first two derivatives are multiplication by $\partial_tU_{\sigma,t}$ and $\partial_t^2U_{\sigma,t}$, including at $t=0$. Their bounds and the Taylor remainder bounds are uniform for $0<\ee\le\ee_0$ in these graph norms.
\end{enumerate}
\end{lemma}
\begin{proof}
Comparison with the harmonic oscillator gives the two-sided eigenvalue estimate. For the gap we use \cite[Theorem 7.6]{DM}. By \cite[Definitions 6.1 and 7.5]{DM}, the class $\mathcal P_1(K)$ consists of the functions $V:\R\to[0,\infty)$ with $V(0)=0$ and $V(-y)\le K V(y)$ for all $y\in\R$ that are positive and $C^1$ on $\R\setminus\{0\}$ and satisfy $K^{-1}V(y)\le yV'(y)\le K V(y)$ for $y\ne0$. For $V\in\mathcal P_1(K)$, let $E_1(V)<E_2(V)<\cdots$ be the eigenvalues of $-\partial_x^2+V$. The estimate in \cite[Theorem 7.6]{DM} is

\[
E_n(V)-E_m(V)\asymp_K(n-m)\frac{E_n(V)}{n},\qquad 1\le m\le n.
\]
Under $\xi=\sqrt\ee\,x$,
\[
\ee^{-1}\Ht_{\ee,t}\simeq-\partial_x^2+W(x),\qquad W(x)=\ee^{-1}U_{\sigma,t}(\sqrt\ee\,x).
\]
By \eqref{eq3.8}, $W(0)=0$, $cx^2\le W(x)\le Cx^2$ and $cx^2\le xW'(x)\le Cx^2$ for all $x$, with $c,C$ independent of $\ee$ and $t$. Hence $W\in\mathcal P_1(C/c)$. Since $e_n/\ee=E_{n+1}(W)$, the gap estimate with $n+2$ in place of $n$ and $m=n+1$ gives
\[
\frac{e_{n+1}-e_n}{\ee}\ge c\frac{e_{n+1}/\ee}{n+2}\ge c'>0,
\]
where the last step uses the lower bound $e_{n+1}\ge c\ee(n+2)$.

For $\Ht=-\ee^2\partial_\xi^2+U$, integration by parts gives
\[
\norm{\Ht v}_2^2=\ee^4\norm{v''}_2^2+\norm{Uv}_2^2
+2\ee^2\int U|v'|^2\dd \xi-\ee^2\int U''|v|^2\dd \xi.
\]
By \eqref{eq3.8}-\eqref{eq3.9}, this controls $\xi^2v$, $\ee \xi v'$ and $\ee^2v''$ by $(1+\Ht)v$, uniformly for $\ee\le\ee_0$. Put $D=\ee\partial_\xi$ and, for $s\in\N_0$, write
\[
N_s(v)=\sum_{r+j\le s}\norm{\xi^rD^jv}_2.
\]
The preceding identity, $U\asymp \xi^2$, and $\norm{U''}_\infty\le C$ give
\[
N_2(v)\le C\norm{(1+\Ht)v}_2.
\]
The words of degree one are controlled by the quadratic form identity $\norm{Dv}_2^2+\int U|v|^2=\ip{\Ht v}{v}$. Normal ordering with $[D,\xi]=\ee$ reduces higher words to second-order ones:
\[
N_{s+2}(v)\le C_s\sum_{r+j\le s}N_2(\xi^rD^jv).
\]
For $s=0$ this is immediate. For $s\ge1$, a word in normal order of degree $s+2$ with at least two factors $\xi$ is obtained from $\xi^2(\xi^rD^jv)$. The remaining words are obtained from $D^2(\xi^rD^jv)$, with commutator terms of degree at most $s$. Moreover,
\begin{align*}
[\Ht,\xi^rD^j]=-2\ee r\xi^{r-1}D^{j+1}-\ee^2r(r-1)\xi^{r-2}D^j\\
&-\sum_{\ell=1}^j\binom j\ell\ee^\ell \xi^rU^{(\ell)}D^{j-\ell}.
\end{align*}
Terms with zero coefficients are omitted. Since $|U'|\le C(1+|\xi|)$ and $U^{(\ell)}$ is bounded for $\ell\ge2$, we obtain
\[
\norm{[\Ht,\xi^rD^j]v}_2\le C_sN_s(v),\qquad r+j\le s.
\] The base estimate applied to each word establishes the recurrence
\[
N_{s+2}(v)\le C_s\{N_s((1+\Ht)v)+N_s(v)\}.
\]
Induction with $s=2m-2$ and $1+\Ht\ge1$ proves $N_{2m}(v)\le C_m\norm{(1+\Ht)^mv}_2$. For the reverse bound, Leibniz's rule implies $N_s((1+\Ht)v)\le C_sN_{s+2}(v)$, which can be iterated. Together these bounds prove \eqref{eq3.11}. All constants depend only on $m,\ee_0$ and finitely many uniform constants in \eqref{eq3.8}-\eqref{eq3.9}.

These estimates also identify the domains of the powers of $1+\Ht$. If $w,(1+\Ht)w\in L^2$ distributionally, local elliptic regularity and a real cutoff $\eta_R(\xi)=\eta(\xi/R)$ give
\[
\norm{D(\eta_Rw)}_2^2+\int(1+U)|\eta_Rw|^2\,\dd \xi
=\Re\ip{(1+\Ht)w}{\eta_R^2w}+\norm{(D\eta_R)w}_2^2.
\]
The right side is uniformly bounded. Letting $R\to\infty$ places $w$ in the form domain and then in the operator domain by the form representation theorem. Applying the base estimate to $\eta_Rw$ and using
\[
(1+\Ht)(\eta_Rw)=\eta_R(1+\Ht)w -2(D\eta_R)Dw-(D^2\eta_R)w
\]
extends that estimate to the maximal domain. The commutator formula above extends the recurrence inductively to every operator power: for each word $T$ occurring at the preceding order, $\Ht Tw=T\Ht w+[\Ht,T]w\in L^2$ distributionally, so the estimate on the maximal domain applies to $Tw$. Conversely, Leibniz's rule shows that $N_{2m}(v)<\infty$ implies $v\in D((1+\Ht)^m)$. Approximation by cutoff and mollification shows that $C_c^\infty(\R)$ is dense in these weighted Sobolev norms. In particular, $\Sch(\R)$ is a core for every power.

For $T$ in (iii), Leibniz's rule and \eqref{eq3.9} bound the weighted Sobolev norm of $Tv$ at order $2m$ by that of $v$ at order $2m+r+2+j$. The same argument applies to multiplication by $\partial_t^jU$, its finite $\xi$-derivatives and polynomials in $\xi$. Equation~\eqref{eq3.11} proves the stated graph bound. In a normal-ordered commutator, each kinetic term differentiates a coefficient and adds at most one factor $D$, while each potential term contains a derivative of $U$ and removes at least one factor $D$. Every term has a factor $\ee$. Induction bounds the total polynomial growth plus differential order after $N$ commutators by $r+2+j+N$. A final use of \eqref{eq3.11} on the spectral subspace proves the second bound in (iii). For each fixed finite range of the indices, these expansions use only finitely many derivatives in \eqref{eq3.9}. The graph bounds include \eqref{eq2.13}.

To justify parameter differentiation, we first work on the Schwartz core. For $|t|,|t'|\le t_0$, Taylor's formula in $t$ and \eqref{eq3.9} show that
\[
r_{t,t'}=U_{\sigma,t'}-U_{\sigma,t}-(t'-t)\partial_tU_{\sigma,t}
\quad\text{satisfies}\quad
|\partial_\xi^kr_{t,t'}(\xi)|\le C_k|t'-t|^2(1+|\xi|)^4.
\]
By \eqref{eq3.11} and Leibniz's rule, multiplication by $r_{t,t'}$ maps $D((1+\Ht)^{m+2})$ to $D((1+\Ht)^m)$ with norm at most $C_m|t'-t|^2$, which proves differentiability of $t\mapsto\Ht_{\ee,t}$ between these spaces. To obtain two continuous derivatives on a fixed scale, use $\mathcal H_{m+3}\to\mathcal H_m$ as in (iv). The bound \eqref{eq3.9} for $j=3$ controls all needed $\xi$-derivatives of the third Taylor coefficient by $C(1+|\xi|)^5$. Taylor's formula and (iii) bound the second-order remainder for $\partial_tU$ by $C_m|t'-t|^2$ in $\mathcal L(\mathcal H_{m+3},\mathcal H_m)$; the difference of the second derivatives has norm at most $C_m|t'-t|$. These bounds hold across $t=0$ and establish (iv). For every fixed $k\in\N_0$, the same argument gives $C^k$ dependence from $\mathcal H_{m+r}$ to $\mathcal H_m$ whenever the integer $r$ satisfies $2r\ge k+3$, using \eqref{eq3.9} through $k+1$ parameter derivatives. The resolvent identity on the core extends to the first and second resolvent derivatives between the fixed graph spaces. Integration against an almost analytic extension of a smooth spectral cutoff proves \eqref{eq2.17}-\eqref{eq2.18}. Corollary~\ref{Cor2.5} bounds these derivatives after multiplication on both sides by powers of $1+\Ht$. At $t=0$, these derivatives are taken between the graph spaces specified above. Derivatives in $\ee>0$ are justified in the same way, allowing finite powers of $\ee^{-1}$.
\end{proof}

For the kernel columns, we also need decay in the source point. Let $A_\ee=-\ee^2\partial_\xi^2+U(\xi)$, where $U$ is smooth and $U(\xi)\ge c\xi^2$.
\begin{lemma}\label{Lem3.4}
For $M>1/4$ and $0<\ee\le\ee_0$,
\[
\norm{(1+A_\ee)^{-M}\delta_q}_2\le C_{M,c,\ee_0}\ee^{-1/2}(1+q^2)^{-M+1/4}.
\]
The constant depends on $U$ only through its lower bound.
\end{lemma}
\begin{proof}
By the Trotter product formula and positivity, the heat kernel of $A_\ee$ is bounded by that of $A_\ee^0=-\ee^2\partial_\xi^2+c\xi^2$. With $\omega=\sqrt c$, the latter heat kernel has diagonal
\[
k_{e^{-vA_\ee^0}}(q,q)=\left(\frac{\omega}{2\pi\ee\sinh(2\omega\ee v)}\right)^{1/2}
\exp\left(-\frac\omega\ee q^2\tanh(\omega\ee v)\right).
\]
This identity follows, for instance, by substituting a Gaussian kernel into the heat equation and its initial condition. Since $\sinh x\ge x$ and $\ee^{-1}\tanh(\omega\ee v)\ge c_1\min(v,1)$ uniformly for $\ee\le\ee_0$,
\[
0\le k_{e^{-vA_\ee}}(q,q)\le C\ee^{-1}v^{-1/2}e^{-c_1q^2\min(v,1)}.
\]
The Laplace formula for $(1+A_\ee)^{-2M}$ yields
\[
\norm{(1+A_\ee)^{-M}\delta_q}_2^2
=\frac1{\Gamma(2M)}\int_0^\infty v^{2M-1}e^{-v}k_{e^{-vA_\ee}}(q,q)\dd v
\le C_M\ee^{-1}(1+q^2)^{-2M+1/2}.
\]
The condition $M>1/4$ is the condition for integrability of this bound near $v=0$. The distributional formula follows by approximating the delta function by smooth functions.
\end{proof}

The following formulation allows the affine coordinates to depend on the Fourier parameters. Consider a family on the fixed space $L^2(\R_u)$,
\begin{equation}\label{eq3.12}
\Pt_{h,\theta}=-h^2\partial_u^2+
U_{\sigma,t(\theta)}(q(\theta)+\varrho(\theta)u),\qquad c\le \varrho\le C,
\end{equation}
where the needed derivatives in $a,b$ of $\varrho,t,q$ are polynomially bounded in $\br\eta$, $|t|\le t_0$, and $|q|\asymp|\eta|$ for large $|\eta|$. Here $\partial_{a,b}^k$ denotes any partial derivative of total order $k$ in $a,b$.
\begin{corollary}\label{Cor3.5}
For every $N$, every $k\in\{0,1,2\}$, and $0<h\le1$,
\begin{equation}\label{eq3.13}
\norm{\partial_{a,b}^k[F(\tau^2\Pt_{h,\theta})\delta_0]}_{L^2_\tau L^2_u}
\le C_Nh^{-1/2}\br\eta^{-N}\norm{F}_{\Hs{k}}.
\end{equation}
\end{corollary}
\begin{proof}
At a fixed parameter, the unitary affine change $\xi=q+\varrho u$ takes $\Pt$ to $\Ht_{h\varrho,t}$. Differentiate \eqref{eq3.12} with $u$ fixed, before making the affine change of variables. The parameter derivatives of $\Pt$ become finite sums of
\begin{equation}\label{eq3.14}
c_{rj\ell}(\theta)\xi^r\partial_t^j\partial_\xi^\ell U_{\sigma,t}(\xi),
\end{equation}
with finitely many indices and polynomially bounded scalar coefficients. The derivatives in \eqref{eq3.14} for parameter orders at most two have $r\le2$ and $j+\ell\le2$; the third derivatives needed for Taylor remainders have $r\le3$ and $j+\ell\le3$. We check Hypothesis (H) of Section~\ref{Sec2} for $\Pt_{h,\theta}$, with $h>0$ fixed. Put $\ee=h\varrho\asymp h$. Under the affine change, $h\partial_u$ corresponds to $\ee\partial_\xi$ and $(1+\Pt_{h,\theta})^m$ to $(1+\Ht_{\ee,t})^m$.

\emph{(H1).} By \eqref{eq3.11}, $\|(1+\Pt_{h,\theta})^mv\|_2\asymp_m\sum_{r+j\le2m}\|(q+\varrho u)^r(h\partial_u)^jv\|_2$, locally uniformly in $\theta$. The polynomials $(q+\varrho u)^r$ and $u^r$ span the same spaces, with coefficients depending continuously on $\theta$. The right side is thus equivalent to $\sum_{r+j\le2m}\|u^r(h\partial_u)^jv\|_2$, which does not depend on $\theta$. By Lemma~\ref{Lem3.3}, $\Sch(\R)$ is a core for every power of $1+\Pt_{h,\theta}$.

\emph{(H2).} The potential $u\mapsto U_{\sigma,t(\theta)}(q(\theta)+\varrho(\theta)u)$ is smooth in $(\theta,u)$. Its $\theta$-derivatives of order at most three, together with their $u$-derivatives, are bounded by $C(1+|u|)^{C'}$, locally uniformly in $\theta$, by the formula for $U_{\sigma,t}$ and the bounds on the derivatives of $\varrho,t,q$. By \eqref{eq3.11} and Leibniz's rule, multiplication by such a function maps $\mathcal H_{m+r}$ to $\mathcal H_m$ for $r$ large. Taylor's formula in $\theta$ verifies (H2).

\emph{(H3).} By \eqref{eq3.10}, $e_n\ge c\ee(n+1)$, so $\Tr(1+\Pt)^{-2}\le C\ee^{-1}\le Ch^{-1}$. Consecutive eigenvalues differ by at least $c\ee\ge c'h$, so every interval of length $h$ contains at most $C$ eigenvalues. Thus (H3) holds with $\Lambda(h)$ replaced by $1$.

\emph{(H4).} The derivatives of $\Pt$ in $\theta$ are the multiplication operators in \eqref{eq3.14}. For them, \eqref{eq2.13} holds by \eqref{eq3.11} and Leibniz's rule, with a shift $r_A$ that depends only on the polynomial degree of the weight. Moreover $[\Pt,f]=-h^2(f''+2f'\partial_u)$ for a multiplication operator $f$, so each commutator gains a factor $h$ on spectral subspaces of bounded energy. This is \eqref{eq2.14}; see Lemma~\ref{Lem3.3}.

All constants are polynomial in $\br\eta$ through the coefficients $c_{rj\ell}(\theta)$. There is an exponent $m_0$, independent of $M$, such that for a smooth cutoff $\psi_\tau$,
\begin{equation}\label{eq3.15}
\norm{(\partial^k\psi_\tau(\Pt))(1+\Pt)^M}_{2\to2}
\le C_{M,k}\br\eta^{m_0},
\qquad k\in\{0,1,2\}.
\end{equation}
By \eqref{eq3.14}, each derivative of $\Pt$ has the form $\sum_\iota c_\iota(\theta)T_\iota$. Here the $T_\iota$ belong to a finite set of multiplication operators of the form $\xi^r\partial_t^j\partial_\xi^\ell U_{\sigma,t}(\xi)$, and $|c_\iota(\theta)|\le C\br\eta^{m_1}$. The operators $T_\iota$ depend only on $t$ and satisfy the bounds of Lemma~\ref{Lem3.3} uniformly. Their constants depend on $M$ and on the number of commutators, but the coefficients $c_\iota$ do not. Every term in \eqref{eq2.17}-\eqref{eq2.18} contains at most two such derivatives, so \eqref{eq3.15} holds with $m_0=2m_1$. The same reasoning applies after one $\tau$-derivative of $\psi_\tau$.

The kernel acquires a factor $\varrho$, while the change from $u$ to $\xi$ contributes $\varrho^{-1/2}$ to its $L^2$ norm. The total factor is therefore $\varrho^{1/2}$. Lemma~\ref{Lem3.4} then implies
\begin{equation}\label{eq3.16}
\norm{(1+\Pt)^{-M}\delta_0}_{L^2_u}
=\varrho^{1/2}\norm{(1+\Ht_{h\varrho,t})^{-M}\delta_q}_{L^2_\xi}
\le C_Mh^{-1/2}(1+q^2)^{-M+1/4}.
\end{equation}
Combine \eqref{eq3.15} and \eqref{eq3.16}. Since $m_0$ is independent of $M$, choose $M$ large enough to obtain
\[
\norm{\partial^k\psi_\tau(\Pt)\delta_0}_{H^1_\tau L^2_u}
\le C_Nh^{-1/2}\br\eta^{-N},\qquad k\in\{0,1,2\}.
\]
Differentiate \eqref{eq2.21} with $\psi=1$ near $J$ and apply Lemma~\ref{Lem2.4}, now without $\Lambda(h)$. The remaining polynomial losses in $\br\eta$ are absorbed by increasing $N$. The resulting estimate is \eqref{eq3.13}.
\end{proof}

\section{Comparison with an escaping well}\label{Sec4}
To transfer the estimates of Corollary~\ref{Cor3.5} to the quartic family, we compare $\Ht_{\ee,t}$ with the full operator in the normalized coordinates,
\[
H_{\ee,t}=-\ee^2\partial_\xi^2+V_{\sigma,t}(\xi)^2,
\qquad V_{\sigma,t}(\xi)=\sigma \xi+t\xi^2/2.
\]
For $t\ne0$, the reflection
\[
S_t\xi=-2\sigma/t-\xi
\]
interchanges the two wells and leaves $V_{\sigma,t}$ invariant; see Figure~\ref{Fig1}. We use the same notation for its unitary action on functions.
\begin{figure}[htbp]
\centering
\begin{tikzpicture}[xscale=1.35,yscale=1.4]
  \clip (-6.2,-0.6) rectangle (2.2,2.75);
  \fill[gray!20] (-0.8,0) rectangle (0.8,2.2);
  \fill[gray!20] (-4.8,0) rectangle (-3.2,2.2);
  \draw[->] (-5.6,0) -- (1.7,0) node[right] {$\xi$};
  \draw[thick,domain=-5.25:1.25,samples=120] plot (\x,{(\x+0.25*\x*\x)^2});
  \draw[thick,dashed] plot coordinates {(-2.200,1.885) (-2.150,1.717) (-2.100,1.571)
    (-2.050,1.445) (-2.000,1.337) (-1.950,1.245) (-1.900,1.168) (-1.850,1.104) (-1.800,1.050)
    (-1.750,1.007) (-1.700,0.971) (-1.650,0.943) (-1.600,0.922) (-1.550,0.901) (-1.500,0.879)
    (-1.450,0.854) (-1.400,0.828) (-1.350,0.800) (-1.300,0.770) (-1.250,0.739) (-1.200,0.706)
    (-1.150,0.671) (-1.100,0.636) (-1.050,0.600) (-1.000,0.563) (-0.950,0.525) (-0.900,0.487)
    (-0.850,0.448) (-0.800,0.410) (-0.750,0.371) (-0.700,0.334) (-0.650,0.296) (-0.600,0.260)
    (-0.550,0.225) (-0.500,0.191) (-0.450,0.160) (-0.400,0.130) (-0.350,0.102) (-0.300,0.077)
    (-0.250,0.055) (-0.200,0.036) (-0.150,0.021) (-0.100,0.010) (-0.050,0.002) (-0.000,0.000)
    (0.050,0.003) (0.100,0.011) (0.150,0.024) (0.200,0.044) (0.250,0.071) (0.300,0.104)
    (0.350,0.145) (0.400,0.194) (0.450,0.251) (0.500,0.316) (0.550,0.391) (0.600,0.476)
    (0.650,0.571) (0.700,0.677) (0.750,0.793) (0.800,0.922) (0.850,1.062) (0.900,1.216)
    (0.950,1.382) (1.000,1.562) (1.050,1.757) (1.100,1.967)};
  \draw (0,0.05) -- (0,-0.05) node[below] {$0$};
  \draw (-4,0.05) -- (-4,-0.05) node[below] {$-2\sigma/t$};
  \draw (-2,0.05) -- (-2,-0.05) node[below] {$-\sigma/t$};
  \draw[<->] (-3.0,2.05) to[bend left=12] node[above] {$S_t$} (-1.0,2.05);
  \node[above] at (0,2.2) {$\operatorname{supp}\ell_t$};
  \node[above] at (-4,2.2) {$\operatorname{supp}(\ell_t\circ S_t)$};
\end{tikzpicture}
\caption{The potential $V_{\sigma,t}^2$ (solid) and the comparison potential $U_{\sigma,t}$ (dashed) for $\sigma=1$ and small $t>0$, not to scale. The two potentials agree on the shaded strip around $0$, and $U_{\sigma,t}$ has a single well. The shaded strips are the supports of the cutoffs $\ell_t$ and $\ell_t\circ S_t$ introduced in Step~2 of the proof of Lemma~\ref{Lem4.1}. The reflection $S_t$ interchanges the two wells and the two shaded strips.} \label{Fig1}
\end{figure}
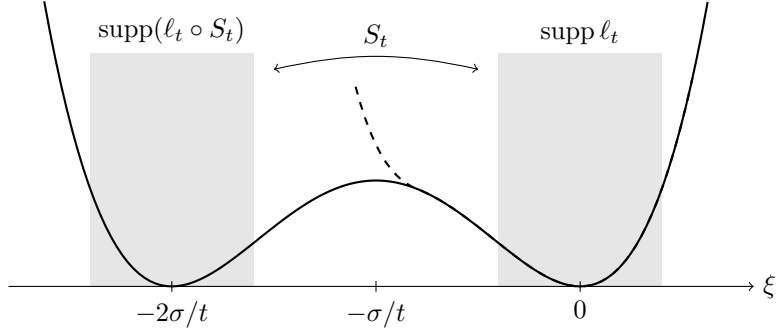
The comparison potential $U_{\sigma,t}$ and the full potential $V_{\sigma,t}^2$ agree when $|t\xi|\le2d_0$.

Fix $J\Subset(0,\infty)$ and $0<\ee_0<\infty$. For $F\in C_c^\infty(J)$, set
\[
\Xi_\tau(\xi,q)=k_{F(\tau^2H_{\ee,t})}(\xi,q)-k_{F(\tau^2\Ht_{\ee,t})}(\xi,q).
\]
\begin{lemma}\label{Lem4.1}
Fix $0<\kappa<1$ and a finite set of indices $(k,m,r,i,j)\in\N_0^5$ with $i+j\le2$. There are $t_0>0$, $c_\kappa>0$, $M<\infty$ and $C$ such that, for $0<\ee\le\ee_0$, $0<|t|\le t_0$, $|tq|\le d_0/2$, and each of these indices,
\begin{equation}\label{eq4.1}
\norm{\xi^k\partial_\xi^m\partial_q^r
\partial_\ee^i\partial_t^j \Xi_\tau(\cdot,q)}_{L^2_\tau L^2_\xi}
\le C\ee^{-M}|t|^{-M}
e^{-c_\kappa/(\ee t^2)}\norm{F}_{\Hs{i+j+\kappa}}.
\end{equation}
The constants depend on $J$, $\ee_0$, $\kappa$, the chosen indices and the cutoffs, but not on $\ee$, $t$ or $q$.
\end{lemma}
\begin{proof}
Write $H=H_{\ee,t}$ and $\Ht=\Ht_{\ee,t}$. Choose $E_*>\sup J$ independently of $\ee,t$ and perform all auxiliary spectral constructions up to energy $4E_*$. After choosing $R_0$ in the weighted estimates below, take $t_0$ so small that $d_0/(2t_0)>2R_0$. The supports of derivatives of the spatial cutoffs then lie in the forbidden region for these energies. The choice of $t_0$ may depend on the chosen indices. The constants $C$ and exponents $M$ may change from one estimate to the next.

\paragraph*{Step 1: Agmon estimates and parameter derivatives} We use the weighted energy method of \cite{Agm} with constants uniform in $\ee,t$ below the fixed energy cutoff.

Put
\[
d_t(\xi)=\min\{|\xi|,|\xi+2\sigma/t|\}.
\]
The roots have separation $2/|t|$. Factoring $V$ shows
\begin{equation}\label{eq4.2}
V_{\sigma,t}(\xi)^2\ge d_t(\xi)^2/4.
\end{equation}
Neumann bracketing at the midpoint $-\sigma/t$ reduces each half-line to a harmonic lower bound centered at its own root. The count for a half-line harmonic oscillator is bounded by the count on the whole line plus one: imposing a Dirichlet condition at the splitting point has codimension one, and Dirichlet-Neumann interlacing bounds the difference by one. Consequently,
\begin{equation}\label{eq4.3}
\operatorname{rank}\ind_{[0,4E_*]}(H_{\ee,t})\le C\ee^{-1}.
\end{equation}
The same bound for $\Ht$ follows from \eqref{eq3.10}.

We use the following weighted energy identity for a real potential $\mathcal V$. If $(-\ee^2\partial_\xi^2+\mathcal V-E)v=f$ and $\Phi$ is Lipschitz, then
\begin{equation}\label{eq4.4}
  \ee^2\norm{(e^{\Phi/\ee}v)'}_2^2
  +\int(\mathcal V-E-|\Phi'|^2)e^{2\Phi/\ee}|v|^2\dd \xi
  =\Re\int e^{2\Phi/\ee}f\overline v\dd \xi.
\end{equation}
For the nonnegative, locally Lipschitz weights below, apply the identity first to $\Phi_L=\min\{\Phi,L\}$ and then let $L\to\infty$. Each $\Phi_L$ is bounded and globally Lipschitz. For the full operator take $\mathcal V=V_{\sigma,t}^2$ and $\Phi=c(d_t^2-R_0^2)_+$. For the comparison operator take $\mathcal V=U_{\sigma,t}$ and $\Phi=c(\xi^2-R_0^2)_+$. Choose $R_0$ large relative to $E_*$ and $c>0$ small. By \eqref{eq4.2} or \eqref{eq3.8}, outside the set where $\Phi=0$,
\[
\mathcal V-E-|\Phi'|^2\ge c'(1+d_t^2)
\quad\text{or}\quad c'(1+\xi^2).
\]
Where $\Phi=0$ the weight is one. Cauchy-Schwarz and absorption in \eqref{eq4.4} give, for example in the full case,
\begin{equation}\label{eq4.5}
\norm{(1+d_t)e^{\Phi/\ee}v}_2
+\ee\norm{(e^{\Phi/\ee}v)'}_2
\le C\bigl(\norm v_2+\norm{e^{\Phi/\ee}f}_2\bigr).
\end{equation}
An identical statement holds with $|\xi|$ for the comparison operator. To pass to higher spatial derivatives, start with the zeroth and first derivative bounds in \eqref{eq4.5} and use $\ee^2v''=(\mathcal V-E)v-f$. Differentiating this equation expresses each higher derivative through previously controlled derivatives of $v$ and $f$, with coefficients of at most polynomial growth. At each iteration, replace the positive coefficient $c$ in the weight by $c/2$ to absorb these coefficients and the required polynomial moments. Only finitely many iterations are needed. In the full case we use $|\xi|\le d_t(\xi)+2/|t|$, so the additional losses are polynomial in $\ee^{-1}$ and $|t|^{-1}$. At distance comparable to $|t|^{-1}$ from both roots, the weighted derivative norms are then bounded by
\begin{equation}\label{eq4.6}
C\ee^{-M}|t|^{-M}e^{-c/(\ee t^2)}.
\end{equation}
The one-dimensional Sobolev inequality supplies pointwise estimates with additional polynomial factors. Terms such as $e^{cR_0^2/\ee}$ that arise when evaluating the weight at a cutoff are absorbed into a smaller positive coefficient of $1/(\ee t^2)$ by reducing $t_0$.

We also need parameter derivatives of the eigenfunctions $\varphi_n$ of $\Ht$ with $e_n\le4E_*$. Since the eigenvalues are simple, these eigenfunctions can be chosen locally in $(\ee,t)$ to be real and normalized. For $\mu\in\{\ee,t\}$,
\[
(\Ht-e_n)\partial_\mu\varphi_n
=-(\partial_\mu\Ht-\partial_\mu e_n)\varphi_n,
\qquad \ip{\partial_\mu\varphi_n}{\varphi_n}=0.
\]
The reduced inverse has norm at most $C\ee^{-1}$ by \eqref{eq3.10}. It commutes with powers of $1+\Ht$. Lemma~\ref{Lem3.3} supplies polynomial bounds in $\ee^{-1}$ for the derivative in every required graph norm. The derivative $\partial_\ee\Ht=-2\ee\partial_\xi^2$ contributes one additional factor $\ee^{-1}$. The eigenvalue derivative is obtained by pairing with $\varphi_n$.

Fix finitely many parameter derivatives. Write $\theta=(\ee,t)$, $\varphi_\alpha=\partial_\theta^\alpha\varphi_n$, $e_\alpha=\partial_\theta^\alpha e_n$, and $A_\beta=\partial_\theta^\beta\Ht$. For $|\alpha|\ge1$,
\begin{equation}\label{eq4.7}
(\Ht-e_n)\varphi_\alpha
=-\sum_{0<\beta\le\alpha}\binom\alpha\beta
(A_\beta-e_\beta)\varphi_{\alpha-\beta}.
\end{equation}
The nonzero derivatives of the term $-\ee^2\partial_\xi^2$ are $A_{(1,0)}=-2\ee^{-1}(\ee\partial_\xi)^2$ and $A_{(2,0)}=-2\ee^{-2}(\ee\partial_\xi)^2$. The remaining nonzero $A_\beta$ are multiplication operators given by $\partial_t^{\beta_2}U$. All of these map a sufficiently high graph space into any given lower graph space, with a polynomial loss in $\ee^{-1}$.

Pairing \eqref{eq4.7} with $\varphi_n$ gives
\[
  e_\alpha=\langle A_\alpha\varphi_n,\varphi_n\rangle
  +\sum_{0<\beta<\alpha}\binom\alpha\beta
        \langle(A_\beta-e_\beta)\varphi_{\alpha-\beta},\varphi_n\rangle.
\]
Here $0<\beta<\alpha$ means $0\le\beta\le\alpha$ componentwise, with $\beta\ne0$ and $\beta\ne\alpha$; all indices are nonnegative. Differentiated normalization determines the parallel component,
\[
  \langle\varphi_\alpha,\varphi_n\rangle
  =-\frac12\sum_{0<\beta<\alpha}\binom\alpha\beta
                \langle\varphi_\beta,\varphi_{\alpha-\beta}\rangle.
\]
The orthogonal component is obtained by applying the reduced inverse of $\Ht-e_n$ to the right side of \eqref{eq4.7}. This inverse commutes with $1+\Ht$ and has norm at most $C\ee^{-1}$. Induction in $|\alpha|$, using \eqref{eq3.11}, establishes, for all $m,\alpha$,
\[
  |e_\alpha|+\|(1+\Ht)^m\varphi_\alpha\|_2
  \le C_{m,\alpha}\ee^{-M_{m,\alpha}},\qquad e_n\le4E_*.
\]

To retain exponential decay after parameter differentiation, use the weights $\Phi_{(s)}(\xi)=s(\xi^2-R_0^2)_+$ with $s>0$.

\begin{claim}\label{Cl4.2}
For every finite set of indices $(\alpha,r,j)$ there are $s_1>0$, $M$ and $C$ such that
\[
  \|e^{\Phi_{(s_1)}/\ee}\xi^r\partial_\xi^j\varphi_\alpha\|_2\le C\ee^{-M},
  \qquad 0<\ee\le\ee_0,\quad |t|\le t_0,\quad e_n\le4E_*.
\]
\end{claim}
\begin{proof}[Proof of Claim~\ref{Cl4.2}]
The phase $\Phi_{(s_1)}$ is independent of $\ee$ and $t$. For $0<s_1<s_2$ and every integer $r\ge0$,
\[
  \sup_\xi(1+|\xi|)^r
    \exp\!\bigl(-[\Phi_{(s_2)}(\xi)-\Phi_{(s_1)}(\xi)]/\ee\bigr)
  \le C_{r,s_1,s_2,\ee_0},\qquad0<\ee\le\ee_0.
\]
A smaller weight coefficient absorbs any polynomial in $\xi$. Applying \eqref{eq4.5} to \eqref{eq4.7} at fixed parameter values controls $e^{\Phi_{(s_1)}/\ee}\varphi_\alpha$ and its first scaled spatial derivative whenever the lower parameter derivatives on its right side have been controlled with a slightly larger coefficient. The term involving $e_\alpha\varphi_n$ is already known from the scalar identity above.

For the spatial derivatives, put $D_\ee=\ee\partial_\xi$. If $(\Ht-e)v=f$, then
\[
  D_\ee^{j+2}v
  =(U-e)D_\ee^jv
    +\sum_{r=1}^j\binom jr\ee^r U^{(r)}D_\ee^{j-r}v
    -D_\ee^jf.
\]
The coefficients have polynomial growth uniformly in $t$. Starting with the bounds for the weighted derivatives of order zero and one in \eqref{eq4.5}, this identity bounds all required spatial derivatives after finitely many further decreases of $s_1$. The preceding exponential inequality controls every polynomial moment. The right side of \eqref{eq4.7} involves only lower parameter derivatives and at most two additional spatial derivatives. Include these lower parameter derivatives at the higher spatial orders required on the right side; the resulting collection is still finite. We order the finitely many pairs $(\alpha,j)$ that are needed, first by $|\alpha|$ and then by $j$, and let $K$ be their number. Choose $s^{(0)}>s^{(1)}>\cdots>s^{(K)}>0$. At step $k$ we use the weight $\Phi_{(s^{(k)})}$ and the bounds already obtained with the weights $\Phi_{(s^{(k')})}$, $k'<k$. After $K$ steps the claim holds with $s_1=s^{(K)}$.
\end{proof}

On the cutoff region $|\xi|\asymp|t|^{-1}$, a reduction of $t_0$ ensures $\Phi_{(s_1)}(\xi)\ge c/t^2$, which proves \eqref{eq4.6}. The one-dimensional Sobolev inequality supplies the corresponding pointwise estimates, including the indicated derivatives.

\paragraph*{Step 2: Quasimodes and spectral clusters of rank two} The quasimodes are constructed by the two-well comparison method of \cite{HS,Sim}.

Choose $\ell\in C_c^\infty((-1,1))$, with $0\le\ell\le1$ and $\ell=1$ on $[-1/2,1/2]$, and set $\ell_t(\xi)=\ell(t\xi/d_0)$. Define
\[
  w_n^L=\frac{\ell_t\varphi_n}{\norm{\ell_t\varphi_n}_2},
   \qquad w_n^R=S_tw_n^L.
\]
These two vectors have disjoint supports and norm one. The modified and full potentials agree on $\supp\ell_t$, so
\[
  (H-e_n)\ell_t\varphi_n=[-\ee^2\partial_\xi^2,\ell_t]\varphi_n.
\]
The residual is supported where $|\xi|\asymp|t|^{-1}$. The residual and its required parameter derivatives satisfy \eqref{eq4.6} in every graph norm. The normalization differs from one by an error of the same kind. Differentiating $S_tf$ in $t$ produces only factors polynomial in $|t|^{-1}$: for a smooth function $f$ of the parameters and of $\xi$,
\[
  \partial_t(S_tf)
  =S_t(\partial_tf)+\frac{2\sigma}{t^2}S_t(\partial_\xi f).
\]
A second application involves only finitely many parameter and spatial derivatives, with polynomial factors in $|t|^{-1}$.

We next show that every eigenvalue of $H$ below $3E_*$ lies near an eigenvalue of $\Ht$. Define
\[
  A_0=\Ht\oplus\Ht,\qquad P_*=\ind_{[0,3E_*]}(H),\qquad
   (J_tf)(\xi)=\bigl(\ell_t(\xi)f(\xi),\ \ell_t(\xi)f(S_t\xi)\bigr).
\]
The cutoff supports at the two wells are disjoint. Apply the estimates for $H$ in the forbidden region to an orthonormal eigenbasis of $\Ran P_*$. Together with \eqref{eq4.3}, these estimates imply
\begin{equation}\label{eq4.8}
  \norm{(I-J_t^*J_t)P_*}
  +\norm{(A_0J_t-J_tH)P_*}\le \mathcal E,
  \qquad \mathcal E=C\ee^{-M}|t|^{-M}e^{-c/(\ee t^2)}.
\end{equation}
For the second term, use the cutoff commutator on each component. For an orthonormal eigenbasis $(\varphi_j)$ of $\Ran P_*$ and either error operator $T$ above,
\[
  \norm{TP_*}\le\norm{TP_*}_{\mathrm{HS}}
  =\left(\sum_{j=1}^{\operatorname{rank}P_*}\norm{T\varphi_j}_2^2\right)^{1/2}.
\]
By \eqref{eq4.3}, the sum contributes only a polynomial factor in $\ee^{-1}$. Apply the same inequality with $(1+H)^mT$ for $T=I-J_t^*J_t$, and with $(1+A_0)^mT$ for $T=A_0J_t-J_tH$, to obtain the respective graph norm estimates. The exponential in $\mathcal E$ absorbs every fixed polynomial loss, uniformly for $0<\ee\le\ee_0$, as $t_0$ decreases. In particular, shrink $t_0$ so that $\mathcal E\le1/4$.

If $Hf=Ef$, $E\le3E_*$, and $\norm f=1$, then
\[
  \norm{J_tf}\ge(1-\mathcal E)^{1/2},\qquad
  \norm{(A_0-E)J_tf}\le \mathcal E.
\]
By the spectral theorem, $E$ lies within $2\mathcal E$ of some $e_n$. This holds for every eigenvalue of $H$ in $[0,3E_*]$.

Let $c_{\mathrm{gap}}\ee$ be a lower bound for the gaps between adjacent eigenvalues of $\Ht$, see \eqref{eq3.10}. Choose $r_*>0$ with $r_*<c_{\mathrm{gap}}/4$ and $r_*\ee_0<E_*/8$. Shrink $t_0$ further so that $C\mathcal E/\ee\le1/4$ and $2\mathcal E<r_*\ee/2$. For $e_n\le5E_*/2$, put
\[
  I_n=(e_n-r_*\ee,e_n+r_*\ee),\quad
  \Pi_n=\ind_{I_n}(H),\quad \Pi_n^0=\ind_{\{e_n\}}(A_0).
\]
The second condition on $r_*$ ensures $e_n+r_*\ee<3E_*$, so $\Pi_n\le P_*$ and \eqref{eq4.8} applies on $\Ran\Pi_n$. Let $W_n:\C^2\to L^2$ have columns $w_n^L,w_n^R$. The residual estimate implies
\[
  \norm{(1-\Pi_n)W_n}\le(r_*\ee)^{-1}\norm{(H-e_n)W_n}
  \le C\mathcal E/\ee.
\]
Thus $\Pi_nW_n$ is injective and $\operatorname{rank}\Pi_n\ge2$.

Conversely, spectral inclusion and separation of the $e_n$ imply $\norm{(H-e_n)\Pi_n}\le2\mathcal E$. For $f\in\Ran\Pi_n$, \eqref{eq4.8} implies
\[
  \norm{(A_0-e_n)J_tf}\le C\mathcal E\norm f,
  \qquad\norm{(1-\Pi_n^0)J_tf}\le C\mathcal E\ee^{-1}\norm f.
\]
Together with the approximate isometry, this shows that $\Pi_n^0J_t$ is injective on $\Ran\Pi_n$. Since $\Pi_n^0$ has rank two, $\Pi_n$ has rank at most two. Each relevant spectral cluster consists of exactly two eigenvalues. Every eigenvalue of $H$ below $2E_*$ belongs to one of these clusters, and distinct clusters have separation at least $c\ee$.

\paragraph*{Step 3: Matrix comparison and parameter derivatives} We use the projections of rank two to obtain an orthonormal frame for each spectral pair from the quasimodes. Define
\[
  G_n=W_n^*\Pi_nW_n,\qquad
   U_n=\Pi_nW_nG_n^{-1/2},\qquad \mathbf C_n=U_n^*HU_n.
\]
The map $U_n$ is an isometry onto $\Ran\Pi_n$. In every graph norm, including all needed parameter derivatives,
\begin{equation}\label{eq4.9}
  U_n-W_n=O\bigl(\ee^{-M}|t|^{-M}e^{-c/(\ee t^2)}\bigr),
  \quad
  \mathbf C_n-e_nI=O\bigl(\ee^{-M}|t|^{-M}e^{-c/(\ee t^2)}\bigr).
\end{equation}
Fix a parameter with $t\ne0$ and $\ee>0$. The verification of (H1) in the proof of Lemma~\ref{Lem3.2}, applied at $(\eta,a,b)=(0,\sigma,t)$ and $h=\ee$, identifies the graph norms with the corresponding weighted derivative norms for a quartic potential. On a sufficiently small neighborhood, both $|t|$ and $\ee$ stay bounded away from zero, so the constants are locally uniform also as $\ee$ varies. Thus the domains $D((1+H)^m)$ are parameter independent on this neighborhood and their norms are locally equivalent. Parameter derivatives are taken between these spaces. Write $R_z=(z-H)^{-1}$. For fixed $z$ and $\mu,\nu\in\{\ee,t\}$,
\[
  \partial_\mu R_z=R_z(\partial_\mu H)R_z,\qquad
  \partial_\nu\partial_\mu R_z=R_z(\partial_\nu\partial_\mu H)R_z
  +R_z(\partial_\mu H)R_z(\partial_\nu H)R_z+R_z(\partial_\nu H)R_z(\partial_\mu H)R_z,
\]
as operators from $D((1+H)^{m+r})$ to $D((1+H)^m)$ for $r$ large enough. The circle $z(\omega)=e_n+r_*\ee e^{i\omega}$, $0\le\omega\le2\pi$, encloses $I_n\cap\Spec(H)$ and stays at distance at least $c\ee$ from $\Spec(H)$. The exact identity
\[
  (z-H)^{-1}W_n=(z-e_n)^{-1}W_n
   +(z-e_n)^{-1}(z-H)^{-1}(H-e_n)W_n
\]
and $(z-e_n)^{-1}dz=i\,d\omega$ give
\[
  \Pi_nW_n-W_n=\frac1{2\pi}\int_0^{2\pi}R_{z(\omega)}(H-e_n)W_n\dd\omega.
\]
The integrand depends on $(\ee,t)$ through $H$, through $z(\omega)$ and through $(H-e_n)W_n$. Along the circle, $\partial_\mu R_z=R_z(\partial_\mu H-\partial_\mu z)R_z$, where $\partial_\mu z=\partial_\mu e_n+r_*e^{i\omega}\partial_\mu\ee$. Every derivative of order at most two of the integrand is a finite sum of terms
\[
  R_zT_1R_z\cdots T_jR_z\,\partial^\beta\bigl[(H-e_n)W_n\bigr],
  \qquad 0\le j\le2,\quad|\beta|\le2,
\]
where each $T_i$ is $\partial_\mu H-\partial_\mu z$ or $\partial_\nu\partial_\mu H-\partial_\nu\partial_\mu z$. Every such term contains the residual $(H-e_n)W_n$ or one of its derivatives. The derivatives of $e_n$ have the polynomial bounds established above. The graph estimates needed between resolvents have the form
\begin{align*}
  \norm{(1+H)^m(z-H)^{-1}v}_2
  &\le C\ee^{-1}\norm{(1+H)^mv}_2,\\
  \norm{(1+H)^mAv}_2
  &\le C\ee^{-M}|t|^{-M}\norm{(1+H)^{m+r_A}v}_2,
\end{align*}
where $A$ is a parameter derivative of $H$. For the second estimate, completing the square yields $V_{\sigma,t}(\xi)=\frac t2(\xi+\sigma/t)^2-\frac1{2t}$. Apply the polynomial graph argument in the proof of Lemma~\ref{Lem3.2} at $(\eta,a,b)=(0,\sigma,t)$ and $h=\ee$. Its loss parameter is $L=1+|t|^{-1}$, which accounts for the factor $|t|^{-M}$. This applies to the $t$-derivatives of $H$. With $D_\ee=\ee\partial_\xi$, the $\ee$-derivatives are $-2\ee^{-1}D_\ee^2$ and $-2\ee^{-2}D_\ee^2$; the same normal-ordering argument applies to $D_\ee^2$, with these additional powers of $\ee^{-1}$. The range $1\le\ee\le\ee_0$ changes only the constants, and the derivatives of $z(\omega)$ are bounded by those of $e_n$. Because the residual is small in every graph norm, these estimates control any finite differentiated resolvent product. The map $G_n\mapsto G_n^{-1/2}$ is analytic near the identity, so the normalization preserves these bounds. Finally use
\[
  \mathbf C_n-e_nI
   =G_n^{-1/2}W_n^*\Pi_n(H-e_n)W_nG_n^{-1/2}
\]
to obtain the second assertion in \eqref{eq4.9}, with derivatives.

Reflection commutes with $H$ and $\Pi_n$ and swaps the columns of $W_n$. Since all data are real, it also swaps the columns of $U_n$, and
\[
  \mathbf C_n=\begin{pmatrix}\mathsf a_n&\mathsf b_n\\\mathsf b_n&\mathsf a_n\end{pmatrix},
   \qquad E_n^\pm=\mathsf a_n\pm \mathsf b_n.
\]
The fixed matrix $2^{-1/2}\bigl(\begin{smallmatrix}1&1\\1&-1\end{smallmatrix}\bigr)$ diagonalizes $\mathbf C_n$. Both $E_n^\pm-e_n$, with all required derivatives, satisfy \eqref{eq4.6}. No reciprocal of $E_n^+-E_n^-$ appears.

\paragraph*{Step 4: Assembly of the kernel and Sobolev estimates} Write $U_n^L,U_n^R$ for the columns of $U_n$. For $q$ with $|tq|\le d_0/2$,
\[
  w_n^L(q)=\frac{\varphi_n(q)}{\norm{\ell_t\varphi_n}_2},\qquad w_n^R(q)=0.
\]
The difference $w_n^L(q)-\varphi_n(q)$ and its derivatives obey \eqref{eq4.6}. The corresponding assertions for $U_n$ hold up to errors of that size, by \eqref{eq4.9} and the pointwise graph estimates. Moreover $U_n^L-\varphi_n$ is exponentially small in all needed weighted derivative norms, whereas $U_n^R$ has polynomial bounds. These statements include $q$-derivatives.

Put
\[
  F^+_{n,\tau}=\tfrac12\bigl(f_\tau(E_n^+)+f_\tau(E_n^-)\bigr),\qquad
  F^-_{n,\tau}=\tfrac12\bigl(f_\tau(E_n^+)-f_\tau(E_n^-)\bigr).
\]
The contribution of the $n$th pair to the full kernel is
\begin{align*}
  k_{n,\tau}(\xi,q)=F^+_{n,\tau}\{U_n^L(\xi)U_n^L(q)+U_n^R(\xi)U_n^R(q)\}+F^-_{n,\tau}\{U_n^L(\xi)U_n^R(q)+U_n^R(\xi)U_n^L(q)\}.
\end{align*}
All the eigenfunctions and columns of $U_n$ are real. Subtracting the contribution of $\Ht$ gives
\begin{align}
  &k_{n,\tau}(\xi,q)-f_\tau(e_n)\varphi_n(\xi)\varphi_n(q)\notag\\
  &\quad=(F^+_{n,\tau}-f_\tau(e_n))\varphi_n(q)U_n^L(\xi)
             +F^-_{n,\tau}\varphi_n(q)U_n^R(\xi)\notag\\
  &\qquad+f_\tau(e_n)\varphi_n(q)(U_n^L(\xi)-\varphi_n(\xi))\notag\\
  &\qquad+F^+_{n,\tau}\{(U_n^L(q)-\varphi_n(q))U_n^L(\xi)
                               +U_n^R(q)U_n^R(\xi)\}\notag\\
  &\qquad+F^-_{n,\tau}\{U_n^R(q)U_n^L(\xi)
                         +(U_n^L(q)-\varphi_n(q))U_n^R(\xi)\}.
  \label{eq4.10}
\end{align}
Every term after the first two on the right of \eqref{eq4.10} contains an amplitude that satisfies \eqref{eq4.6}, together with all derivatives needed. In the second term, the factor $F^-_{n,\tau}$ is small in $L^2_\tau$ by \eqref{eq2.11} and controls the polynomially bounded column $U_n^R$. Since $U_n^R$ is concentrated where $|\xi|\lesssim|t|^{-1}$, the weight $\xi^k$ acting on it produces only a factor $C|t|^{-k}$.

For either parameter $\mu\in\{\ee,t\}$,
\begin{align*}
\partial_\mu\{f_\tau(E_n^\pm)-f_\tau(e_n)\}
=[f_\tau'(E_n^\pm)-f_\tau'(e_n)]\partial_\mu e_n+f_\tau'(E_n^\pm)\partial_\mu(E_n^\pm-e_n).
\end{align*}
For $\mu,\nu\in\{\ee,t\}$, a second derivative gives
\begin{align*}
  &\partial_\nu\partial_\mu\{f_\tau(E_n^\pm)-f_\tau(e_n)\}\\
  &\quad=[f_\tau''(E_n^\pm)-f_\tau''(e_n)]
                          \partial_\mu e_n\partial_\nu e_n\\
  &\qquad+[f_\tau'(E_n^\pm)-f_\tau'(e_n)]\partial_\mu\partial_\nu e_n\\
  &\qquad+f_\tau''(E_n^\pm)\{
        \partial_\mu(E_n^\pm-e_n)\partial_\nu E_n^\pm
        +\partial_\mu e_n\partial_\nu(E_n^\pm-e_n)\}\\
  &\qquad+f_\tau'(E_n^\pm)\partial_\mu\partial_\nu(E_n^\pm-e_n).
\end{align*}
Only derivatives of $f_\tau$ of order at most two occur. The derivatives of $e_n$ are polynomially bounded, and every derivative of $E_n^\pm-e_n$ satisfies \eqref{eq4.6}. The differences of scalar derivatives are controlled by \eqref{eq2.11}. The remaining terms are bounded by \eqref{eq2.10} using their exponentially small amplitudes. Since
\[
  F^+_{n,\tau}-f_\tau(e_n)
  =\tfrac12\sum_{\pm}(f_\tau(E_n^\pm)-f_\tau(e_n)),\qquad
  F^-_{n,\tau}=\tfrac12\{f_\tau(E_n^+)-f_\tau(E_n^-)\},
\]
these estimates control the first two terms in \eqref{eq4.10}. In every other term, Leibniz's rule leaves an exponentially small amplitude and uses at most $i+j$ scalar derivatives of $f_\tau$. Derivatives in $\xi,q$ and multiplication by $\xi^k$ act only on the amplitudes.

At each parameter choose a finite set, locally independent of the parameter, containing every index $n$ for which $e_n$ or $E_n^\pm$ lies in a neighborhood of the support of some $f_\tau$. Outside that set the terms and their indicated derivatives vanish. After enlarging $J_0$ by a fixed margin and reducing $t_0$, all active $e_n,E_n^\pm$ lie in a common compact subinterval of $(0,\infty)$. Step 2 places every eigenvalue of $H$ contributing to $f_\tau(H)$ in one of the pairs. Step 3 supplies uniform estimates for each pair. There are at most $C\ee^{-1}$ such indices. Summing with Lemma~\ref{Lem2.3} and the triangle inequality establishes \eqref{eq4.1}, with Sobolev order $i+j+\kappa$. Raising the exponentially small differences $|E_n^\pm-e_n|$ to the power $\kappa$ only decreases the positive coefficient in the exponential.
\end{proof}

\begin{remark}\label{Rem4.3}
Lemma~\ref{Lem4.1} compares the kernel columns $k(\cdot,q)$ for $q$ in the well near zero, with $\xi$ ranging over the whole line. The restriction on the source variable reflects the contribution of the escaping well. To see this, fix $\ee>0$, $\tau\in[1,2]$ and $n\ge0$ such that $F(\tau^2\ee(2n+1))\ne0$. The normalized vector $w_n^R$ is a quasimode of $H_{\ee,t}$ at energy $e_n(\ee,t)$, while its projection onto the spectral subspace of $\Ht_{\ee,t}$ associated with $\supp f_\tau$ tends to zero by the Agmon estimates in Claim~\ref{Cl4.2}. Since $e_n(\ee,t)\to\ee(2n+1)$, testing the multiplier difference on $w_n^R$ gives
\[
  \liminf_{t\to0}\norm{F(\tau^2H_{\ee,t})-F(\tau^2\Ht_{\ee,t})}_{2\to2}
  \ge |F(\tau^2\ee(2n+1))|>0.
\]
The same quasimodes also explain the failure of norm resolvent convergence to $H_{\ee,0}$. They converge weakly to zero, and the harmonic oscillator $H_{\ee,0}$ has compact resolvent. Thus $(H_{\ee,0}+1)^{-1}w_n^R\to0$, whereas $(H_{\ee,t}+1)^{-1}w_n^R=(e_n(\ee,t)+1)^{-1}w_n^R+o_{L^2}(1)$.
\end{remark}

We apply Lemma~\ref{Lem4.1} and Corollary~\ref{Cor3.5} to $P_{h,\theta}$ by expressing its coefficients through an affine change of variables. The argument below remains valid with $\Omega_0$ replaced throughout by any fixed compact annulus in $\R^2\setminus\{(0,0)\}$; its constants, including $\delta_*$, may depend on that annulus. In particular, it applies on the enlarged annulus $\Omega'$ of Section~\ref{Sec5}. On the annulus $\Omega_0$ in the $(a,b)$-plane, take a small $\delta_*>0$ and consider
\begin{equation}\label{eq4.11}
  |b|\br\eta\le2\delta_*.
\end{equation}
Here $|a|\ge c>0$. Set
\[
  \varrho=(a^2-2\eta b)^{1/2},\qquad
   \sigma=\sgn a,\qquad t=b/\varrho^2,\qquad
   q=\frac{2\eta \varrho}{a+\sigma \varrho}.
\]
The denominator is bounded away from zero, $c\le \varrho\le C$, $|q|\asymp|\eta|$, $|tq|\le C\delta_*$, and every $a,b$ derivative of finite order is polynomially bounded in $\br\eta$. Fix $\delta_1>0$ so that the bounds $c\le\varrho\le C$ hold with fixed constants for every $0<\delta_*\le\delta_1$. Using this lower bound in the energy estimate following \eqref{eq4.13}, choose $h_0\ge1$ so that the multipliers vanish for $h\ge h_0$, set $\ee_0=h_0C$, take $t_0$ from Lemma~\ref{Lem4.1}, and finally choose $\delta_*\le\delta_1$ small enough to satisfy the conditions below. For $b\ne0$, $q=\varrho(a-\sigma \varrho)/b$. The displayed formula is smooth across $b=0$. The defining relations imply
\[
  \eta=\sigma q+tq^2/2,\qquad a=\varrho(\sigma+tq),\qquad b=\varrho^2t,
   \qquad p_\theta(u)=V_{\sigma,t}(q+\varrho u).
\]
The affine change $\xi=q+\varrho u$ takes $P_{h,\theta}$ to $H_{h\varrho,t}$. Choose $\delta_*$ so that $|t|\le t_0$ and $|tq|\le d_0/2$. The exact column identity is
\[
  k_{F(\tau^2P_{h,\theta})}(u,0)
     =\varrho\,k_{F(\tau^2H_{h\varrho,t})}(q+\varrho u,q).
\]
For the parameter derivatives, put $\Psi_\tau(u)=\varrho \Xi_\tau(q+\varrho u,q;h\varrho,t)$. Keeping $h,u$ fixed gives
\begin{align*}
  \partial_a\Psi_\tau=\varrho_a\Xi_\tau+\varrho\bigl[(q_a+\varrho_au)\partial_\xi\Xi_\tau+q_a\partial_q\Xi_\tau
              +h\varrho_a\partial_\ee \Xi_\tau+t_a\partial_t\Xi_\tau\bigr].
\end{align*}
All terms on the right are evaluated at $(q+\varrho u,q;h\varrho,t)$. The formula for $\partial_b$ is identical with $a$ replaced by $b$. A second application of the product rule produces second derivatives involving at most two parameter derivatives of $\Xi_\tau$, so Lemma~\ref{Lem4.1} applies to every term, including derivatives in both variables. The scalar coefficients and powers of $u=(\xi-q)/\varrho$ in this expansion have polynomial growth in $\br\eta$. In \eqref{eq4.11}, $\br\eta\le C/|t|$ when $t\ne0$. Since $\ee=h\varrho$ and $c\le \varrho\le C$, each comparison error is bounded by a fixed polynomial in $h^{-1},|t|^{-1}$ times $e^{-c_1/(ht^2)}$. For all $A,B,N\ge0$,
\[
  h^{-A}|t|^{-B}e^{-c_1/(ht^2)}
  \le C_{A,B,N}h^{-1/2}|t|^N
  \le C'_{A,B,N}h^{-1/2}\br\eta^{-N},
  \qquad 0<h\le1,\quad0<|t|\le t_0.
\]
Here $e^{-c_1x}\le C_Lx^{-L}$ for $x>0$, and one takes $L\ge\max\{A-1/2,(B+N)/2,0\}$ with $x=(ht^2)^{-1}$. Combining Lemma~\ref{Lem4.1} with Corollary~\ref{Cor3.5} yields, for $k+l\le2$,
\begin{equation}\label{eq4.12}
  \norm{\partial_a^k\partial_b^l[F(\tau^2P_{h,\theta})\delta_0]}_{L^2_\tau L^2_u}
  \le C_{N,\kappa}h^{-1/2}\br\eta^{-N}
                \norm{F}_{\Hs{k+l+\kappa}},\qquad 0<h\le1.
\end{equation}
Now consider the parameters with $b=0$, that is, $t=0$. There $H_{\ee,0}=\Ht_{\ee,0}=-\ee^2\partial_\xi^2+\xi^2$, so the comparison error vanishes. Let $F\in C_c^\infty(J)$. By \eqref{eq2.3} and the Schwartz calculus, as in the proof of Lemma~\ref{Lem2.1}, the column $F(\tau^2P_{h,\theta})\delta_0$ is a $C^2$ function of $(a,b)$ with values in $L^2_u$, also across $b=0$, uniformly for $1\le \tau\le2$. By Lemma~\ref{Lem3.3} and the proof of Corollary~\ref{Cor3.5}, the same holds for the column for the comparison operator $F(\tau^2\Pt_{h,\theta})\delta_0$, since $t$, $q$ and $\varrho$ depend smoothly on $(a,b)$ across $b=0$. The difference $\Psi_\tau$ is consequently $C^2$ across $b=0$. Since these derivatives are continuous in $(a,b)$ uniformly in $\tau$, their $L^2_\tau L^2_u$ norms are continuous in $(a,b)$. For $t\ne0$, the derivatives of $\Psi_\tau$ of order at most two satisfy the bound above, which tends to zero as $t\to0$. By continuity these derivatives vanish at $t=0$. Estimate \eqref{eq4.12} holds throughout \eqref{eq4.11}.

We next show that this region contributes nothing for sufficiently large $h$. If a quadratic polynomial has two distinct real roots at which the derivative has absolute value $\varrho>0$, its explicit inverse branches imply, for every $A>0$,
\begin{equation}\label{eq4.13}
  |\{u:|p(u)|\le A\}|\le CA/\varrho.
\end{equation}
For $b\ne0$ and $A<|p(-a/b)|$, the length is given by a difference of two square roots. For $A\ge|p(-a/b)|$, it is $O(\sqrt{A/|b|})\le CA/\varrho$. In the linear case $b=0$ the set is an interval of length $2A/|a|=2A/\varrho$. If $v$ is normalized and has form energy $E$ for $P_{h,\theta}$, then at least $3/4$ of its mass lies in $|p|\le2\sqrt E$, while
\[
  \norm v_\infty^2\le2\norm v_2\norm{v'}_2\le2\sqrt E/h.
\]
Together with \eqref{eq4.13}, this implies $E\ge ch\varrho$. Thus $F(\tau^2P_{h,\theta})=0$ in \eqref{eq4.11} when $h$ is large enough. Local parameter derivatives vanish there as well.
\section{Weighted kernel estimates}\label{Sec5}
The parameter estimates from Sections~\ref{Sec3} and~\ref{Sec4} give spatial moments of the kernels through \eqref{eq2.7}. We begin by integrating over an annulus in the $(a,b)$-plane and using homogeneity to pass from the average in $\tau$ to a fixed multiplier. All constants in this section depend on the interval $J$.
\begin{proposition}\label{Prop5.1}
Let $\Omega\Subset\R^2\setminus\{(0,0)\}$ be an annulus. For $k,l\in\N_0$ with $k+l\le2$ and $\kappa>0$,
\begin{equation}\label{eq5.1}
  \int_{\R\times\Omega}\norm{\partial_a^k\partial_b^l v_F(h,\theta)}_2^2\dd\theta
  \le C h^{-1}\Lambda(h)^4\norm F_{\Hs{k+l+\kappa}}^2,
  \qquad 0<h\le1.
\end{equation}
For $h\ge1$, the left side is at most $C_{N,\kappa}h^{-N}\norm F_{\Hs{k+l+\kappa}}^2$ for every $N$. When $k=l=0$, the small-$h$ estimate holds without $\Lambda(h)^4$ and without $\kappa$.
\end{proposition}
\begin{proof}
To pass from the averaged estimates to a fixed multiplier, we use homogeneity. Define
\[
  D_\tau(\eta,a,b)=(\tau\eta,\tau^2a,\tau^3b),\qquad
  U_\tau v(u)=\tau^{1/2}v(\tau u).
\]
Then
\[
  P_{h,D_\tau\theta}=\tau^2U_\tau P_{h,\theta}U_\tau^{-1},\qquad
   U_\tau^{-1}\delta_0=\tau^{1/2}\delta_0.
\]
For $f_\tau(E)=F(\tau^2E)$, the differentiated identity is
\[
  (\partial_a^k\partial_b^l v_F)(h,D_\tau\theta)
   =\tau^{1/2-2k-3l}U_\tau\partial_a^k\partial_b^l
         [f_\tau(P_{h,\theta})\delta_0].
\]
Since $U_\tau$ is unitary, the squared column norm contributes $\tau^{1-4k-6l}$, while the Jacobian of $D_\tau$ is $\tau^6$. Their product is uniformly bounded for $1\le \tau\le2$. An enlarged annulus $\Omega'$ contains the $(a,b)$-projection of $D_\tau^{-1}(\R\times\Omega)$ for every $1\le \tau\le2$. Changing variables for each $\tau$ and averaging, we obtain
\begin{equation}\label{eq5.2}
  \int_{\R\times\Omega}\norm{\partial_a^k\partial_b^l v_F(h,\theta)}_2^2\dd\theta
  \le C\int_{\R\times\Omega'}\int_1^2
  \norm{\partial_a^k\partial_b^l[f_\tau(P_{h,\theta})\delta_0]}_2^2
  \dd \tau\dd\theta.
\end{equation}

Split the parameter integral measurably according to $b^2(1+\eta^2)/\delta_*^2$, assigning each parameter to $|b|\br\eta\le2\delta_*$ or its complement. The complement lies in $|b|\br\eta\ge c\delta_*$, so the two estimates below cover this disjoint decomposition. Apply \eqref{eq4.12} pointwise on the first region. On the second, $|b|^{-1}\le C\br\eta$, and Lemmas~\ref{Lem3.1} and~\ref{Lem3.2} verify Hypothesis (H) with constants polynomial in $\br\eta$. Choose a smooth $\psi$ equal to one near $J$ and differentiate
\[
  f_\tau(P)\delta_0=f_\tau(P)\psi(\tau^2P)\delta_0.
\]
The vectors $\psi_\tau(P)\delta_0$ and their derivatives obey Lemma~\ref{Lem2.1}, including one $\tau$-derivative, and their rapid decay in $\eta$ absorbs every polynomial loss. The derivatives in Lemma~\ref{Lem3.2} include both pure second derivatives, so the same argument applies to $(k,l)=(2,0)$ and $(0,2)$. For every $k+l\le2$ this yields
\[
  \norm{\partial_a^k\partial_b^l[f_\tau(P_{h,\theta})\delta_0]}_{L^2_\tau L^2_u}
  \le C_Nh^{-1/2}\Lambda(h)^{k+l}\br\eta^{-N}
                    \norm F_{\Hs{k+l+\kappa}}.
\]
Since these bounds hold pointwise in $\theta$, we split the $\theta$-integral into the two regions. The decay in $\eta$ absorbs the polynomial losses in $\br\eta$; integrating over this decomposition, we obtain \eqref{eq5.1}.

For large $h$, the operator $P_{h,\theta}$ has no spectrum in $J_0$ on the first region. Indeed, the energy bound following \eqref{eq4.13}, together with the uniform positive lower bound for $\varrho$ there, shows that, with $h_0$ as fixed after \eqref{eq4.11}, $F(\tau^2P_{h,\theta})=0$ for $h\ge h_0$ and $1\le\tau\le2$, and its local parameter derivatives vanish as well. On the second region apply the proof of Lemma~\ref{Lem2.4} using intervals of unit length. The constants in the trace, eigenvalue counting, and graph estimates are polynomial in $h$ and $\br\eta$. For example, the number of eigenvalues in a bounded interval is at most $Ch^{-1}|b|^{-1/2}$, and is thus polynomial in $\br\eta$. The finitely many commutators have polynomial $h$ losses by Lemma~\ref{Lem3.2}. On $\Omega'$,
\[
  h|a|+h^2|b|\ge ch,\qquad h\ge1.
\]
The decay in \eqref{eq2.8} absorbs any polynomial in $h$ and $\eta$. For $1\le h\le h_0$, retain the comparison argument on the first region, with $\ee_0$ as fixed after \eqref{eq4.11} and using the part of the proof of Lemma~\ref{Lem2.4} that uses the trace bound, applied to the comparison operator. On the second region, use the same argument in place of the small-$h$ cluster estimate. The constants may depend on $h_0$: the comparison error still decays rapidly in $\eta$, while the remaining trace and graph constants grow polynomially in $\br\eta$ and are absorbed by the column decay. Enlarging the constant by a factor depending on $h_0$ and $N$ yields the asserted large-$h$ estimate throughout $h\ge1$.

For $k=l=0$, neither derivative terms nor cluster bounds are needed. At a fixed $P$, we use its eigenbasis, the Sobolev inequality \eqref{eq2.12}, and Lemma~\ref{Lem2.3} with $k=0$ on $f_\tau(P)\psi(\tau^2P)\delta_0$. Lemma~\ref{Lem2.1} bounds the $L^2_\tau L^2_u$ norm by $Ch^{-1/2}\br\eta^{-N}\norm F_2$; substitution into \eqref{eq5.2} establishes the assertion for $k=l=0$ without logarithmic or Sobolev loss.
\end{proof}

Formula~\eqref{eq2.7} converts Proposition~\ref{Prop5.1} into estimates for the integer moments of each localized kernel.
\begin{proposition}\label{Prop5.2}
For $k,l\in\N_0$ with $k+l\le2$ and $\kappa>0$,
\[
  \norm{\br z^k\br w^lK_{F,h}}_2
   \le C_\kappa h^{3/2-k-2l}\Lambda(h)^2
            \norm F_{\Hs{k+l+\kappa}},\qquad 0<h\le1.
\]
For every $N$,
\[
  \norm{\br z^k\br w^lK_{F,h}}_2
  \le C_{N,\kappa}h^{-N}\norm F_{H^{k+l+\kappa}},\qquad h\ge1.
\]
When $k=l=0$, one also has
\[
  \norm{K_{F,h}}_2\le Ch^{3/2}\norm F_2,\qquad 0<h\le1.
\]
\end{proposition}
\begin{proof}
Apply \eqref{eq2.7}, differentiate the cutoff $\vartheta$ by Leibniz's rule, and use Proposition~\ref{Prop5.1}. The square root of the scaling factor is
\[
  \bigl(h^{4-2k-4l}h^{-1}\bigr)^{1/2}=h^{3/2-k-2l}.
\]
The inhomogeneous weights also involve lower moments. For $h\le1$, their estimates are bounded by the displayed estimate, and the required Sobolev norms of $F$ are bounded by $C\norm F_{H^{k+l+\kappa}}$. For $h\ge1$ every additional power of $h$ is absorbed into the arbitrarily rapid decay.
\end{proof}

Interpolation of the integer moments and summation of the localized kernels over $h$ lead to the full kernel estimate.
\begin{proposition}\label{Prop5.3}
If $\beta,\gamma\ge0$ and $\beta+2\gamma<3/2$, then for every $\kappa>0$,
\begin{equation}\label{eq5.3}
  \norm{\br z^\beta\br w^\gamma K_F}_2
  \le C\norm F_{\Hs{\beta+\gamma+\kappa}}.
\end{equation}
\end{proposition}
\begin{proof}
To interpolate on the full Sobolev scale, apply the estimates to the linear operator $F\mapsto K_{\psi F,h}$, where $\psi$ is a smooth cutoff equal to one on $J$. The three cases $(k,l)=(0,0),(2,0),(0,2)$ of Proposition~\ref{Prop5.2} have Sobolev orders $\kappa,2+\kappa,2+\kappa$. For $\beta,\gamma\ge0$ and $\beta+\gamma\le2$, the interpolation coefficients are
\[
  1-\frac{\beta+\gamma}{2},\qquad\frac\beta2,\qquad\frac\gamma2.
\]
If $\beta+\gamma>0$, first interpolate between the two estimates of order two with coefficients $\beta/(\beta+\gamma)$ and $\gamma/(\beta+\gamma)$, and then with the estimate of order zero. Including the immediate case $\beta=\gamma=0$, we obtain
\begin{equation}\label{eq5.4}
  \norm{\br z^\beta\br w^\gamma K_{F,h}}_2
  \le C h^{3/2-\beta-2\gamma}\Lambda(h)^2
           \norm F_{\Hs{\beta+\gamma+\kappa}},\qquad h\le1.
\end{equation}
Complex interpolation with these coefficients combines the weights geometrically and the Sobolev orders and powers of $h$ affinely. The large-$h$ estimates interpolate in the same way and retain arbitrarily rapid decay.

The condition $\beta+2\gamma<3/2$ implies $\beta+\gamma<2$, so the interpolation applies. The sum of \eqref{eq5.4} over dyadic $h\le1$ converges because $3/2-\beta-2\gamma>0$. The sum for $h\ge1$ converges by the rapid decay estimate. By \eqref{eq2.5} and Euclidean Plancherel, the unweighted $L^2$ limit of the dyadic sum is $K_F$. The omitted set $(\zeta,\lambda)=(0,0)$ is Lebesgue null. The weighted limit has the same distributional limit, proving \eqref{eq5.3}.
\end{proof}

We use the Plancherel formula to identify kernels obtained by approximation and to estimate each piece of the wave representation when introducing the weight in $\rho$.
\begin{lemma}\label{Lem5.4}
There is $c_{\calL}>0$ such that
\begin{equation}\label{eq5.5}
  \norm{K_F}_2^2=c_{\calL}\int_0^\infty |F(E)|^2E^{5/2}\dd E.
\end{equation}
The identity holds for $F\in L^2((0,\infty),E^{5/2}\dd E)$, with $K_F$ defined by approximation in $L^2(G)$. Moreover, $\ind_{\{0\}}(\calL)=0$, and $\ind_B(\calL)=0$ for every Lebesgue null Borel subset $B$ of $(0,\infty)$.
\end{lemma}
\begin{proof}
Formula~\eqref{eq5.5} is established in \cite[Proposition 1.10]{BBGL}. It also follows from the Plancherel measure for homogeneous sub-Laplacians on stratified groups in \cite[Proposition 4.2]{MRT}. For the Engel group the density is $c\lambda^{7/2-1}=c\lambda^{5/2}$. To include the vanishing of spectral projections of null sets, let $\operatorname{tr}_G$ be the canonical faithful normal semifinite trace on the right group von Neumann algebra, normalized by
\[
  \operatorname{tr}_G(T_k^*T_k)=\norm k_2^2,\qquad T_kf=f*k.
\]
In the present coordinates it can equivalently be obtained from \eqref{eq2.4} by integrating the ordinary operator traces against $(2\pi)^{-3}\,d\nu\,d\lambda$. The Schwartz heat kernel and the semigroup identity give $0<\operatorname{tr}_G(e^{-\calL})<\infty$. The inequality $\ind_{(0,E]}(\calL)\le e^E e^{-\calL}$ shows that $N(E)=\operatorname{tr}_G(\ind_{(0,E]}(\calL))$ is finite. For $\mathcal D_rf(g)=r^{7/2}f(\delta_rg)$, one has $\mathcal D_r^{-1}\calL\mathcal D_r=r^2\calL$. The kernel of $\mathcal D_rT_k\mathcal D_r^{-1}$ is $r^7k(\delta_r g)$, whose squared $L^2$ norm is $r^7\norm k_2^2$. The trace identity, extended to positive operators by monotone approximation, implies $N(r^2E)=r^7N(E)$, so $N(E)=cE^{7/2}$. The spectral projection at zero vanishes: if $\calL u=0$ in $L^2$, the Dirichlet form implies $Xu=Yu=0$. Commutators yield $Zu=Wu=0$ distributionally, so $u$ is constant on the connected group. Since Haar measure is infinite, $u=0$. The positive heat trace ensures $c>0$.

Integration against $\dd N$ establishes \eqref{eq5.5}, first for simple functions and then by spectral approximation, extending $F\mapsto K_F$ continuously from smooth compactly supported multipliers to $L^2((0,\infty),E^{5/2}\dd E)$. The measure $\dd N$ has density $c_{\calL}E^{5/2}$. For every Lebesgue null Borel set $B$ in a bounded positive interval, the projection $\ind_B(\calL)$ has trace zero, and faithfulness implies that this projection vanishes. A countable union of such intervals covers $(0,\infty)$.
\end{proof}

We use the Plancherel identity to identify the unweighted $L^2$ limit, then extend the weighted estimates and convolution identities to Sobolev multipliers.
\begin{lemma}\label{Lem5.5}
Let $J\Subset(0,\infty)$, let $J'\Subset(0,\infty)$ be a compact interval containing a neighborhood of $J$, let $r\ge0$, and let $\omega\ge1$ be a measurable weight on $G$. Suppose that $\norm{\omega K_F}_2\le C_0\norm F_{H^r}$ for all $F\in C_c^\infty(J')$. Let $F\in H^r(\R)$ with $\supp F\subset J$, and let $K_F$ be defined by Lemma~\ref{Lem5.4}. Then $\norm{\omega K_F}_2\le C_0\norm F_{H^r}$, and $\omega K_{F_n}\to\omega K_F$ in $L^2(G)$ for suitable $F_n\in C_c^\infty(J')$ with $F_n\to F$ in $H^r$. If moreover $r>1/2$ and $\omega^{-1}\in L^2(G)$, then $K_{F_n}\to K_F$ in $L^1(G)$ and $F(\calL)f=f*K_F$ for $f\in L^2(G)$. In particular, identities between distributional derivatives of the kernels pass to the limit.
\end{lemma}
\begin{proof}
Let $F_n$ be the convolution of $F$ with a standard mollifier supported in $[-1/n,1/n]$. For large $n$, $F_n\in C_c^\infty(J')$ and $F_n\to F$ in $H^r$, so $(\omega K_{F_n})$ is a Cauchy sequence in $L^2(G)$. By \eqref{eq5.5}, $K_{F_n}\to K_F$ in $L^2(G)$. A subsequence converges almost everywhere, identifying the weighted limit as $\omega K_F$ and establishing the first two assertions. If $\omega^{-1}\in L^2(G)$, the Cauchy-Schwarz inequality upgrades convergence to $L^1(G)$. If moreover $r>1/2$, then $F_n\to F$ uniformly, so $F_n(\calL)f\to F(\calL)f$ in $L^2$ by the spectral theorem, while $f*K_{F_n}\to f*K_F$ in $L^2$ by Young's inequality.
\end{proof}

The weight in $\rho$ is obtained by decomposing the wave representation into dyadic time intervals and applying \eqref{eq5.5} to each piece.
\begin{lemma}\label{Lem5.6}
For $N\ge0$ and $\kappa>0$,
\begin{equation}\label{eq5.6}
  \norm{\br\rho^N K_F}_2
  \le C_{N,\kappa,J}\norm F_{\Hs{N+\kappa}},
  \qquad F\in C_c^\infty(J).
\end{equation}
\end{lemma}
\begin{proof}
Put $m(r)=F(r^2)$ and extend evenly to $\R$. Choose an even dyadic partition $\sum_{j\ge0}\varphi_j(t)=1$, with $\varphi_0$ supported in a bounded interval and, for $j\ge1$, $\varphi_j$ supported where $|t|\asymp2^j$ and satisfying $|\partial_t^a\varphi_j|\le C_a2^{-ja}$. Set $\widehat m_j=\varphi_j\widehat m$.

Set $D_Hf=(Xf,Yf)$ and consider the self-adjoint closure of the first-order block operator
\[
  \mathcal D_H=\begin{pmatrix}0&D_H^*\\D_H&0\end{pmatrix},
\]
initially defined on compactly supported smooth vectors. Its square has scalar block $D_H^*D_H=\calL$, and its control distance is the horizontal control distance. Compactness of closed control balls ensures completeness; essential self-adjointness follows from \cite[Corollary 6.2]{CM}. By \cite[Theorem 7.4 and Section 7.2]{CM}, $\cos(t\sqrt{\calL})$ propagates at speed one. By the cosine representation of the even multiplier, the convolution kernel of $m_j(\sqrt{\calL})$ is supported in $\rho\le C2^j$.

Changing variables $E=r^2$ in \eqref{eq5.5}, we obtain
\[
  \norm{K_{m_j(\sqrt{\cdot})}}_2
  \asymp\norm{r^3m_j(r)}_{L^2(\R)}
  =C\norm{\partial_t^3(\varphi_j\widehat m)}_2.
\]
For finite sums, the triangle inequality now implies
\begin{align*}
  \norm{\br\rho^N K_F}_2
  &\le C\sum_{j\ge0}2^{jN}\norm{\partial_t^3(\varphi_j\widehat m)}_2\\
  &\le C\sum_{\ell=0}^3\sum_{j\ge0}
        2^{j(N-3+\ell)}\norm{\ind_{\supp\varphi_j}\partial_t^\ell\widehat m}_2.
\end{align*}
Since $\ell\le3$, the exponent $N-3+\ell$ is at most $N$. Cauchy-Schwarz in $j$, using a summable factor $2^{-j\kappa}$ and the bounded overlap of the dyadic shells in $t$, bounds this by
\[
  C_\kappa\sum_{\ell=0}^3
  \norm{(1+|t|)^{N+\kappa}\partial_t^\ell\widehat m}_2
  \asymp C_\kappa\sum_{\ell=0}^3\norm{r^\ell m}_{\Hs{N+\kappa}}.
\]
The support of $m$ is contained in a compact subset of $\R\setminus\{0\}$. Multiplication by $r^\ell$ there and the smooth change of variables $E=r^2$ are bounded on these Sobolev spaces, establishing \eqref{eq5.6} and convergence of the finite sums in weighted $L^2$. Their spectral $L^2$ limit identifies the kernel as $K_F$.
\end{proof}

\begin{proof}[Proof of Theorem~\ref{Thm1.2}]
We combine the weights in $z,w$ from Proposition~\ref{Prop5.3} with the weight in $\rho$ from Lemma~\ref{Lem5.6}. For $\alpha=0$, Proposition~\ref{Prop5.3} applies. Otherwise choose $0<\mu<1$ sufficiently small that
\[
  \beta_0=\frac{\beta}{1-\mu},\qquad
  \gamma_0=\frac{\gamma}{1-\mu},\qquad
  \beta_0+2\gamma_0<\frac32.
\]
This is possible because $\beta+2\gamma<3/2$. Put $N=\alpha/\mu$. Interpolate the operator $F\mapsto K_{\psi F}$ between \eqref{eq5.3} with exponents $(\beta_0,\gamma_0)$ and \eqref{eq5.6} with exponent $N$. The target weight is
\[
  (\br z^{\beta_0}\br w^{\gamma_0})^{1-\mu}
        (\br\rho^N)^\mu
  =\br\rho^\alpha\br z^\beta\br w^\gamma.
\]
With the same arbitrarily small loss $\kappa$ in both estimates, the resulting Sobolev order is
\[
  (1-\mu)(\beta_0+\gamma_0+\kappa)
        +\mu(N+\kappa)=\alpha+\beta+\gamma+\kappa.
\]
Thus \eqref{eq1.3} holds for $F\in C_c^\infty(J)$; the same argument on a slightly larger interval $J'$, followed by Lemma~\ref{Lem5.5}, extends it to $F\in H^{\alpha+\beta+\gamma+\kappa}(\R)$ with $\supp F\subset J$.
\end{proof}

\begin{remark}\label{Rem5.7}
The proof of Theorem~\ref{Thm1.1} uses Theorem~\ref{Thm1.2} with $\beta,\gamma<1/2$. The condition $\beta+2\gamma<3/2$ ensures dyadic summability in Proposition~\ref{Prop5.3}; the question of its optimality is stated in Section~\ref{Sec1}.
\end{remark}

For a fixed smooth multiplier, the localized kernels admit an asymptotic expansion that determines their leading moments. Define
\begin{align*}
  \check\vartheta(z,w)&=(2\pi)^{-2}\int_{\R^2}
        e^{i(za+wb)}\vartheta(a,b)\dd a\dd b,
  \\
  k_F^0(x,y)&=(2\pi)^{-2}\int_{\R^2}
        e^{i(x\xi+y\eta)}F(\xi^2+\eta^2)\dd\xi\dd\eta.
\end{align*}
Here $k_F^0$ is the kernel of $F(-\Delta)$ on the abelian quotient $\R^2$, and the same cutoff formula defines $K_{F,h}$ for every $h>0$. The notation $O_{\Sch}(h^2)$ means a remainder bounded by $C h^2$ in each Schwartz seminorm on $\R^4$, with $C$ depending on that seminorm and on $F$ and $\vartheta$.

\begin{proposition}\label{Prop5.8}
For fixed $F\in C_c^\infty(J)$, as $h\to0^+$,
\begin{equation}\label{eq5.7}
  h^{-3}K_{F,h}(x,y,z/h,w/h^2)
  =k_F^0(x,y)\left(\check\vartheta(z,w)
        -\frac{hxy}{2}\partial_z\check\vartheta(z,w)\right)
        +O_{\Sch}(h^2).
\end{equation}
For every $k,l\in\N_0$,
\begin{equation}\label{eq5.8}
  h^{2k+4l-3}\norm{z^kw^lK_{F,h}}_2^2
  =\frac{\norm{\partial_a^k\partial_b^l\vartheta}_2^2}{16\pi^3}
        \norm F_2^2+O_{F,\vartheta,k,l}(h^2).
\end{equation}
\end{proposition}
\begin{proof}
By the Schwartz functional calculus, the partial Fourier transform
\[
  A_F(x,y,\zeta,\lambda)=\int_{\R^2}
        e^{-i(z\zeta+w\lambda)}K_F(x,y,z,w)\dd z\dd w
\]
is a Schwartz function. Formula~\eqref{eq2.3} at $\zeta=\lambda=0$, followed by inversion in $\eta$, shows that $A_F(x,y,0,0)=k_F^0(x,y)$.

For real $F$, the real coefficients of $\calL$ imply that $K_F$ is real, and self-adjointness implies $K_F(g^{-1})=K_F(g)$. Decomposing $F$ into its real and imaginary parts proves this identity for complex $F$ as well; see also \cite{BBGL}. The measure-preserving automorphism
\[
  (x,y,z,w)\longmapsto(-x,-y,z,-w)
\]
preserves $\calL$ and its convolution kernels. Since
\[
  (x,y,z,w)^{-1}=(-x,-y,xy-z,-w+xz-x^2y/2),
\]
the two symmetries give
\[
  K_F(x,y,z,w)=K_F(x,y,xy-z,w-xz+x^2y/2).
\]
For fixed $x,y$, changing variables in the absolutely convergent first moment yields
\[
  2\int_{\R^2}zK_F(x,y,z,w)\dd z\dd w=xy\,k_F^0(x,y),
  \qquad
  \partial_\zeta A_F(x,y,0,0)=-\frac{i xy}{2}k_F^0(x,y).
\]

By the definition of the localized kernel,
\[
  h^{-3}K_{F,h}(x,y,z/h,w/h^2)
  =\FA^{-1}_{a,b}\bigl[\vartheta(a,b)A_F(x,y,ha,h^2b)\bigr](z,w).
\]
Because $\vartheta$ has compact support, its product with $A_F(x,y,ha,h^2b)$ is a smooth function of $h$ near zero with values in the Schwartz space. Taylor's formula reads
\[
  \vartheta(a,b)A_F(x,y,ha,h^2b)
  =\vartheta(a,b)k_F^0(x,y)
        -\frac{i hxy}{2}a\vartheta(a,b)k_F^0(x,y)
        +O_{\Sch}(h^2).
\]
Partial Fourier inversion and $\FA^{-1}_{a,b}(a\vartheta)=-i\partial_z\check\vartheta$ prove \eqref{eq5.7}.

Write the left side of \eqref{eq5.7} as $R_h$. A change of variables shows that
\[
  h^{2k+4l-3}\norm{z^kw^lK_{F,h}}_2^2
  =\norm{z^kw^lR_h}_2^2.
\]
The term linear in $h$ in the squared norm vanishes, since $k_F^0$ is radial and $\int_{\R^2}xy|k_F^0(x,y)|^2\dd x\dd y=0$. The Schwartz remainder is $O(h^2)$ in every weighted $L^2$ norm. Finally, Euclidean Plancherel identifies the leading factors:
\[
  \norm{k_F^0}_2^2=\frac{1}{4\pi}\norm F_2^2,
  \qquad
  \norm{z^kw^l\check\vartheta}_2^2
        =\frac{1}{4\pi^2}\norm{\partial_a^k\partial_b^l\vartheta}_2^2,
\]
which proves \eqref{eq5.8}.
\end{proof}

The cutoff $\vartheta$ is nonzero by \eqref{eq2.5}. For every $k,l\in\N_0$, compact support and successive integration in $a$ and $b$ show that $\partial_a^k\partial_b^l\vartheta$ is also nonzero. Thus, for nonzero $F$, the leading constant in \eqref{eq5.8} is positive and the localized moment has order $h^{3/2-k-2l}$. The powers of $h$ in Proposition~\ref{Prop5.2} are attained by every fixed nonzero smooth multiplier, up to the factor $\Lambda(h)^2$ in that upper bound.

\section{The multiplier theorem and sharpness}\label{Sec6}
Weighted Cauchy-Schwarz converts Theorem~\ref{Thm1.2} into $L^1$ kernel estimates when the reciprocal weight lies in $L^2$. Let
\[
  \omega=\br\rho^\alpha\br z^\beta\br w^\gamma,
  \qquad 0\le\beta,\gamma<1/2.
\]
On the annulus $R\le\rho<2R$, $R\ge1$, use $g=\delta_Rg'$. Since
\[
  \int_{|z'|\le C}(1+R^2|z'|)^{-2\beta}\dd z'\le C_\beta R^{-4\beta},
  \qquad
  \int_{|w'|\le C}(1+R^3|w'|)^{-2\gamma}\dd w'\le C_\gamma R^{-6\gamma},
\]
one obtains
\[
  \int_{R\le\rho<2R}\omega^{-2}\dd g
   \le CR^{7-2\alpha-4\beta-6\gamma}.
\]
Summation over dyadic annuli shows that, if $\alpha+2\beta+3\gamma>7/2$, then $\omega^{-1}\in L^2$ and
\begin{equation}\label{eq6.1}
  \norm{\ind_{\{\rho>R\}}\omega^{-1}}_2\le CR^{7/2-\alpha-2\beta-3\gamma}.
\end{equation}

\begin{corollary}\label{Cor6.1}
Let $s>2$ and $0\le\delta<s-2$. Then
\[
  \norm{\br\rho^\delta K_F}_1
  \le C_{s,\delta,J}\norm F_{H^s},
  \qquad F\in C_c^\infty(J).
\]
The estimate extends to $F\in H^s(\R)$ supported in $J$.
\end{corollary}
\begin{proof}
To keep $\beta,\gamma$ positive, choose $0<\kappa<1/4$ so that $2+\delta+5\kappa<s$, and set
\[
  \beta=\gamma=\frac12-\kappa,\qquad
  \alpha=1+\delta+6\kappa.
\]
These parameters satisfy $\beta+2\gamma<3/2$ and $\alpha+\beta+\gamma+\kappa=2+\delta+5\kappa<s$. Applying \eqref{eq6.1} with $\alpha-\delta$ in place of $\alpha$ shows that $\br\rho^\delta\omega^{-1}\in L^2$, since
\[
  \alpha-\delta+2\beta+3\gamma-\frac72=\kappa>0.
\]
Cauchy-Schwarz and Theorem~\ref{Thm1.2} prove the estimate; Lemma~\ref{Lem5.5}, applied with the weight $\omega$, extends it by the same inequality.
\end{proof}

For the tail and horizontal derivative bounds needed to control kernel differences, fix $s>2$ and set $\delta=(s-2)/2$. Corollary~\ref{Cor6.1} implies
\begin{equation}\label{eq6.2}
  \norm{K_F}_1\le C_{s,\delta}\norm F_{\Hs s},\qquad
  \int_{\rho>R}|K_F|\dd g\le C_{s,\delta}R^{-\delta}\norm F_{\Hs s},
  \quad R\ge1.
\end{equation}

For the horizontal derivatives, let $S=K_\psi$, where $\psi$ is a smooth spectral cutoff equal to one near $J$. Functional calculus implies $K_F=K_F*S=S*K_F$. With the convolution convention \eqref{eq2.2},
\[
  \begin{aligned}
   XK_F&=K_F*(XS),& YK_F&=K_F*(YS),\\
   \Xt K_F&=(\Xt S)*K_F,& \Yt K_F&=(\Yt S)*K_F.
   \end{aligned}
\]
All horizontal derivatives of $S$ are integrable. Young's inequality and \eqref{eq6.2} yield
\begin{equation}\label{eq6.3}
  \norm{XK_F}_1+\norm{YK_F}_1+\norm{\Xt K_F}_1+\norm{\Yt K_F}_1
  \le C_s\norm F_{\Hs s}.
\end{equation}
By Lemma~\ref{Lem5.5}, these identities remain valid distributionally for $F\in H^s(\R)$ with $\supp F\subset J$.

We apply these estimates to the dyadic pieces of a bounded Borel multiplier with finite local Sobolev norm. Choose $\psi_0\in C_c^\infty((1/2,2))$ with $\sum_{j\in\mathbb Z}\psi_0(2^{-j}E)=1$ for $E>0$. Set
\[
  F_j(E)=F(E)\psi_0(2^{-j}E),\qquad
  G_j(u)=F(2^ju)\psi_0(u),\qquad r_j=2^{-j/2}.
\]
It suffices to consider $\mathfrak{M}=\norm F_{\Hs{s},\sloc}<\infty$. One has
\begin{equation}\label{eq6.4}
  \sup_j\norm{G_j}_{\Hs s}\le C\mathfrak{M}.
\end{equation}
To verify this for the arbitrary nonzero cutoff $\chi$ in \eqref{eq1.1}, cover $\supp\psi_0$ by finitely many intervals on which dilates $\chi(a_\ell u)$ do not vanish. With a subordinate smooth partition, write
\[
  \psi_0(u)=\sum_\ell b_\ell(u)\chi(a_\ell u).
\]
The changes $v=a_\ell u$ and multiplication by $b_\ell$ are bounded on $H^s$, proving \eqref{eq6.4}. Sobolev embedding also implies the essential bound $\norm F_{L^\infty((0,\infty))}\le C\mathfrak{M}$. Lemma~\ref{Lem5.4} makes this sufficient for arbitrary bounded Borel representatives.

Let $K_j=K_{F_j}$. By homogeneity,
\[
  K_j(g)=r_j^{-7}K_{G_j}(\delta_{1/r_j}g).
\]
By \eqref{eq6.2} and \eqref{eq6.3},
\begin{equation}\label{eq6.5}
  \int_{\rho>Rr_j}|K_j|\dd g\le CR^{-\delta}\mathfrak{M}\quad(R\ge1),
  \qquad
  \norm{\Xt K_j}_1+\norm{\Yt K_j}_1\le Cr_j^{-1}\mathfrak{M}.
\end{equation}

For any $g_0\in G$, a curve tangent to the right-invariant horizontal fields from the identity to $g_0^{-1}$ of length at most $C\rho(g_0)$ yields
\begin{equation}\label{eq6.6}
  \norm{k(g_0^{-1}\cdot)-k}_1
  \le C\rho(g_0)(\norm{\Xt k}_1+\norm{\Yt k}_1).
\end{equation}
Such a curve is obtained by inverting a curve tangent to the left-invariant horizontal fields. Apply the fundamental theorem of calculus to $k(\gamma(t)g)$. Right translation by $g$ carries each right-invariant field at $\gamma(t)$ to the same field at $\gamma(t)g$. Haar invariance then proves \eqref{eq6.6}. Approximation extends it to integrable weak derivatives.

The tail and translation estimates, applied on the two sides of the spatial scale $\rho(g_0)$, control kernel differences uniformly over finite spectral sums. For a finite set $\mathcal I\subset\mathbb Z$, put $K_{\mathcal I}=\sum_{j\in \mathcal I}K_j$.
\begin{proposition}\label{Prop6.2}
For a sufficiently large constant $C_1$,
\[
  \sup_{g_0\ne e}\int_{\rho(g)>C_1\rho(g_0)}
         |K_{\mathcal I}(g_0^{-1}g)-K_{\mathcal I}(g)|\dd g\le C\mathfrak{M},
\]
with $C$ independent of $\mathcal I$.
\end{proposition}
\begin{proof}
Put $r=\rho(g_0)$ and choose $j_0$ so that $2^{j_0/2}r\le1<2^{(j_0+1)/2}r$. For $j\le j_0$, \eqref{eq6.6} and \eqref{eq6.5} bound $\norm{K_j(g_0^{-1}\cdot)-K_j}_1$ by $Cr2^{j/2}\mathfrak{M}$, whose sum is $O(\mathfrak{M})$. For $j>j_0$, take the absolute values of the two kernel terms separately. The quasi-triangle inequality, for sufficiently large $C_1$, implies
\[
  \rho(g)>C_1r\quad\Longrightarrow\quad\rho(g_0^{-1}g)>cC_1r.
\]
Haar invariance and \eqref{eq6.5} bound their integral by $C(2^{j/2}r)^{-\delta}\mathfrak{M}$, which also has sum $O(\mathfrak{M})$. Restricting either sum to $\mathcal I$ proves the proposition.
\end{proof}

\begin{proof}[Proof of Theorem~\ref{Thm1.1}]
Let $T_{\mathcal I}=\sum_{j\in \mathcal I}F_j(\calL)$, the operator of convolution with $K_{\mathcal I}$. Finite overlap of the spectral cutoffs and \eqref{eq6.4} imply $\norm{T_{\mathcal I}}_{2\to2}\le C\mathfrak{M}$. Let $f_B\in L^1(G)$ be supported in a ball $B$ with center $c_B$, and let $f_B$ have integral zero. By \eqref{eq2.2} and Haar invariance, for $x=c_Bv$,
\[
  T_{\mathcal I}f_B(x)=\int_G f_B(c_Bg_0)\bigl[K_{\mathcal I}(g_0^{-1}v)-K_{\mathcal I}(v)\bigr]\dd g_0.
\]
Proposition~\ref{Prop6.2} implies $\int_{(C_1B)^c}|T_{\mathcal I}f_B|\dd x\le C\mathfrak{M}\norm{f_B}_1$. Together with the uniform $L^2$ bound, the Calder\'on-Zygmund decomposition for the control distance and Haar measure \cite[Chapter III]{CW} yields
\begin{equation}\label{eq6.7}
  \norm{T_{\mathcal I}}_{L^1\to L^{1,\infty}}\le C_s\mathfrak{M},
\end{equation}
uniformly in finite $\mathcal I$.

Take $\mathcal I=\{-N,\ldots,N\}$. The spectral partition and the vanishing of the spectral projection at zero imply $T_{\mathcal I}f\to F(\calL)f$ in $L^2$ for every $f\in L^2$, by dominated convergence in the spectral theorem. For $f\in L^1\cap L^2$, take a subsequence converging almost everywhere and apply Fatou's lemma to the distribution functions to pass \eqref{eq6.7} to $F(\calL)$. Completion in measure extends this operator to $L^1$, consistently with the $L^2$ realization. Interpolation establishes the $L^p$ estimates for $1<p\le2$, and the same result applied to $\overline F$, followed by duality, covers $2<p<\infty$, completing the upper bound in Theorem~\ref{Thm1.1}.

To identify the sharp threshold, we apply \cite[Theorem 1.1(i)]{MMN}. The theorem applies to sub-Laplacians defined by a smooth positive measure on an $n$-dimensional manifold and a smooth function on the cotangent bundle that is a positive semidefinite quadratic form on each fiber, with any self-adjoint extension, provided the form is bracket generating at some point. For the Engel group, the quadratic form sums the squares of a covector's values on $X$ and $Y$. Together with Haar measure it defines $\calL$, and the fields $X,Y$ and their brackets span the tangent space at every point. Thus the hypotheses hold with $n=d=4$.

The theorem states that if $p\in[1,\infty]$ and $a\ge0$ are such that $\norm{m(\calL)}_{L^p\to L^p}\le C\norm m_{L^\infty_{a,\sloc}}$ for all Schwartz functions $m$, where $\norm m_{L^\infty_{a,\sloc}}=\sup_{t\ge0}\norm{\chi m(t\cdot)}_{L^\infty_a(\R)}$ and $L^\infty_a(\R)$ is the Bessel potential space of order $a$, then
\[
  a\ge d\left|\frac1p-\frac12\right|.
\]
The same theorem supplies a lower bound for the Mihlin-H\"ormander threshold defined in \cite{MMN}. This is the infimum of the $a\ge0$ such that, for every $p\in(1,\infty)$, $\norm{m(\calL)}_{L^p\to L^p}\le C\norm m_{L^2_{a,\sloc}}$ for all bounded continuous $m:[0,\infty)\to\C$, where $\norm m_{L^2_{a,\sloc}}=\sup_{t\ge0}\norm{\chi m(t\cdot)}_{H^a(\R)}$.

With the same cutoff, the norm in \eqref{eq1.1} is at most this norm because its supremum is taken over $r>0$. Different cutoffs give equivalent norms. At each order admitted in the definition of $\varsigma(\calL)$, the estimates hold for every bounded Borel function. They imply the estimates in the definition of the threshold in \cite{MMN} for every bounded continuous function. Since the threshold in \cite{MMN} is at least $d/2=2$, we obtain $\varsigma(\calL)\ge2$.

Combined with the upper bound, this establishes \eqref{eq1.2}.
\end{proof}

\section*{Acknowledgments}
The authors thank Professors Thomas Schick and Ingo Witt for their supervision and support during their doctoral studies at the University of G\"ottingen. The authors used GPT-6 (OpenAI) to assist with literature searches and English language editing. All mathematical arguments and proofs in the final manuscript were written and checked by the authors, who take full responsibility for its content.

\medskip
\noindent\textbf{Funding.} This work was supported by the National Natural Science Foundation of China (12301145, 12561020, 12261107) and the Yunnan Fundamental Research Projects (202401AU070123, 202601AT070048).

\medskip
\noindent\textbf{Data availability.} Data sharing is not applicable to this article as no new data were created or analyzed in this study.

\medskip
\noindent\textbf{Conflict of interest.} The authors declare that they have no conflict of interest.

\authoraddresses

\begin{thebibliography}{99}
\setlength{\itemsep}{0pt}

\bibitem{Agm}
S.~Agmon, \emph{Lectures on Exponential Decay of Solutions of Second-Order Elliptic Equations: Bounds on Eigenfunctions of N-Body Schr\"odinger Operators}, Mathematical Notes 29, Princeton University Press, Princeton, NJ, 1982.

\bibitem{BBGL}
H.~Bahouri, D.~Barilari, I.~Gallagher and M.~L\'eautaud, Spectral summability for the quartic oscillator with applications to the Engel group, \emph{J. Spectr. Theory} \textbf{13} (2023), no.~2, 623-706.

\bibitem{Ben}
L.~Benedetto, Quantum-classical correspondence and obstruction to dispersion on the Engel group, \emph{Amer. J. Math.}, to appear.

\bibitem{Cha}
M.~Chatzakou, A note on spectral multipliers on Engel and Cartan groups, \emph{Proc. Amer. Math. Soc.} \textbf{150} (2022), no.~5, 2259-2270.

\bibitem{Chr}
M.~Christ, $L^p$ bounds for spectral multipliers on nilpotent groups, \emph{Trans. Amer. Math. Soc.} \textbf{328} (1991), no.~1, 73-81.

\bibitem{CW}
R.~R.~Coifman and G.~Weiss, \emph{Analyse harmonique non-commutative sur certains espaces homog\`enes. \'Etude de certaines int\'egrales singuli\`eres}, Lecture Notes in Mathematics 242, Springer-Verlag, Berlin, 1971.

\bibitem{CM}
M.~G.~Cowling and A.~Martini, Sub-Finsler geometry and finite propagation speed, in \emph{Trends in Harmonic Analysis}, Springer INdAM Series 3, Springer, Milan, 2013, pp.~147-205.

\bibitem{DMrobust}
G.~M.~Dall'Ara and A.~Martini, A robust approach to sharp multiplier theorems for Grushin operators, \emph{Trans. Amer. Math. Soc.} \textbf{373} (2020), no.~11, 7533-7574.

\bibitem{DM}
G.~M.~Dall'Ara and A.~Martini, An optimal multiplier theorem for Grushin operators in the plane, I, \emph{Rev. Mat. Iberoam.} \textbf{39} (2023), no.~3, 897-974.

\bibitem{DMII}
G.~M.~Dall'Ara and A.~Martini, An optimal multiplier theorem for Grushin operators in the plane, II, \emph{J. Fourier Anal. Appl.} \textbf{28} (2022), no.~2, Paper No.~32, 29~pp.

\bibitem{DS}
M.~Dimassi and J.~Sj\"ostrand, \emph{Spectral Asymptotics in the Semi-Classical Limit}, London Mathematical Society Lecture Note Series 268, Cambridge University Press, Cambridge, 1999.

\bibitem{DOS}
X.~T.~Duong, E.~M.~Ouhabaz and A.~Sikora, Plancherel-type estimates and sharp spectral multipliers, \emph{J. Funct. Anal.} \textbf{196} (2002), no.~2, 443-485.

\bibitem{FS}
G.~B.~Folland and E.~M.~Stein, \emph{Hardy Spaces on Homogeneous Groups}, Mathematical Notes 28, Princeton University Press, Princeton, NJ, 1982.

\bibitem{Heb}
W.~Hebisch, Multiplier theorem on generalized Heisenberg groups, \emph{Colloq. Math.} \textbf{65} (1993), no.~2, 231-239.

\bibitem{HS}
B.~Helffer and J.~Sj\"ostrand, Multiple wells in the semi-classical limit I, \emph{Comm. Partial Differential Equations} \textbf{9} (1984), no.~4, 337-408.

\bibitem{Hor}
L.~H\"ormander, Estimates for translation invariant operators in $L^p$ spaces, \emph{Acta Math.} \textbf{104} (1960), 93-140.

\bibitem{Hul}
A.~Hulanicki, A functional calculus for Rockland operators on nilpotent Lie groups, \emph{Studia Math.} \textbf{78} (1984), no.~3, 253-266.

\bibitem{MarHR}
A.~Martini, Spectral multipliers on Heisenberg-Reiter and related groups, \emph{Ann. Mat. Pura Appl.} (4) \textbf{194} (2015), no.~4, 1135-1155.

\bibitem{MMfree}
A.~Martini and D.~M\"uller, $L^p$ spectral multipliers on the free group $N_{3,2}$, \emph{Studia Math.} \textbf{217} (2013), no.~1, 41-55.

\bibitem{MMGr}
A.~Martini and D.~M\"uller, A sharp multiplier theorem for Grushin operators in arbitrary dimensions, \emph{Rev. Mat. Iberoam.} \textbf{30} (2014), no.~4, 1265-1280.

\bibitem{MMnew}
A.~Martini and D.~M\"uller, Spectral multiplier theorems of Euclidean type on new classes of two-step stratified groups, \emph{Proc. Lond. Math. Soc.} (3) \textbf{109} (2014), no.~5, 1229-1263.

\bibitem{MMtop}
A.~Martini and D.~M\"uller, Spectral multipliers on 2-step groups: topological versus homogeneous dimension, \emph{Geom. Funct. Anal.} \textbf{26} (2016), no.~2, 680-702.

\bibitem{MMN}
A.~Martini, D.~M\"uller and S.~Nicolussi Golo, Spectral multipliers and wave equation for sub-Laplacians: lower regularity bounds of Euclidean type, \emph{J. Eur. Math. Soc.} \textbf{25} (2023), no.~3, 785-843.

\bibitem{MRT}
A.~Martini, F.~Ricci and L.~Tolomeo, Convolution kernels versus spectral multipliers for sub-Laplacians on groups of polynomial growth, \emph{J. Funct. Anal.} \textbf{277} (2019), no.~6, 1603-1638.

\bibitem{MSGr}
A.~Martini and A.~Sikora, Weighted Plancherel estimates and sharp spectral multipliers for the Grushin operators, \emph{Math. Res. Lett.} \textbf{19} (2012), no.~5, 1075-1088.

\bibitem{MauMed}
G.~Mauceri and S.~Meda, Vector-valued multipliers on stratified groups, \emph{Rev. Mat. Iberoam.} \textbf{6} (1990), 141-154.

\bibitem{MSwave}
D.~M\"uller and A.~Seeger, Sharp $L^p$ bounds for the wave equation on groups of Heisenberg type, \emph{Anal. PDE} \textbf{8} (2015), no.~5, 1051-1100.

\bibitem{MS}
D.~M\"uller and E.~M.~Stein, On spectral multipliers for Heisenberg and related groups, \emph{J. Math. Pures Appl.} (9) \textbf{73} (1994), no.~4, 413-440.

\bibitem{Nie}
L.~Niedorf, Spectral multipliers on M\'etivier groups, \emph{Studia Math.} \textbf{282} (2025), no.~2, 149-197.

\bibitem{SimFI}
B.~Simon, \emph{Functional Integration and Quantum Physics}, 2nd ed., AMS Chelsea Publishing, Providence, RI, 2005.

\bibitem{Sim}
B.~Simon, Semiclassical analysis of low lying eigenvalues. II. Tunnelling, \emph{Ann. of Math.} (2) \textbf{120} (1984), no.~1, 89-118.

\end{thebibliography}
\end{document}